\documentclass{article}

\usepackage{amsfonts, amsmath, amssymb, amsthm, amscd}
\usepackage{ascmac}
\usepackage[dvipdfm]{color}
\usepackage{verbatim}
\usepackage[dvips]{graphicx}
\usepackage{graphics}
\usepackage[all,2cell]{xypic}
\usepackage{mathptmx}

\objectmargin+{1mm}
\labelmargin+{0.8mm}
\SelectTips{cm}{12}

\usepackage{helvet}         
\usepackage{courier}        
\usepackage{type1cm} 

\UseAllTwocells

\theoremstyle{plain}  
\newtheorem{thm}{Theorem}[section]

\newtheorem{lem}[thm]{Lemma}
\newtheorem{prop}[thm]{Proposition}

\theoremstyle{definition}
\newtheorem{df}[thm]{Definition}
\newtheorem{ex}[thm]{Example}

\theoremstyle{remark}
\newtheorem{con}[thm]{Conjecture}
\newtheorem{rem}[thm]{Remark}

\title{A survey of Gersten's conjecture}
\date{}

\author{Satoshi Mochizuki}

\begin{document}

\maketitle

\abstract{
This article is the extended notes of 
my survey talk of Gersten's conjecture given 
at the workshop ``Bousfield classes form a set: a workshop 
in a memory of Tetsusuke Ohkawa'' at Nagoya University in August 2015. 
In the last section, I give an explanation of  
my recent work of motivic Gernsten's conjecture. 
}

\section*{Introduction}

\noindent
This article is the extended notes of 
my survey talk of Gersten's conjecture in \cite{Ger73} given 
at the workshop ``Bousfield classes form a set: a workshop 
in a memory of Tetsusuke Ohkawa'' at Nagoya University in August 2015. 
(The slide movie of my talk at the workshop is \cite{Moc15}.) 
In this article, we provide an explanation of a proof 
of Gersten's conjecture in \cite{Moc16a} by 
emphasizing the conceptual idea behind 
the proof rather than its technical aspects. 
In the last section, I give an explanation of  
my recent work of motivic Gernsten's conjecture 
in \cite{Moc16b}.

\paragraph{acknowledgement}
I wish to express my deep gratitude to all 
organizers of the workshop and specially 
professor Norihiko Minami 
for giving me the opportunity to present my work.

\section{What is Gersten's conjecture?}

We start by recalling the following result. 
For $n=0$, $1$, the results are classical and 
for $n=2$, it was given by Spencer Bloch 
in \cite{Blo74} by utilzing the second universal Chern class map in 
\cite{Gro68}.

\begin{prop}
\label{prop:classical Bloch formula}
For a smooth variety $X$ over a field, 
we have the canonical isomorphisms
\begin{equation}
\label{eq:Bloch formula}
\operatorname{CH}^n(X)\simeq \operatorname{H}_{\operatorname{Zar}}^n(X,\mathcal{K}_n)
\end{equation}
for $n=0$, $1$, $2$ 
where $\mathcal{K}_n$ is the Zariski sheafication of the $K$-presheaf 
$U\mapsto K_n(U)$.
\qed
\end{prop}

\noindent
On the other hand, there are the following spectral sequences. 
To give a precise statement, we introduce some notations. 
For a noetherian scheme $X$, 
we write $\mathcal{M}_X$ for the category of coherent sheaves on $X$. 
There is a filtration
\begin{equation}
\label{eq:filtration of MX}
0\subset \cdots \subset \mathcal{M}_X^2\subset \mathcal{M}_X^1\subset
\mathcal{M}_X^0=\mathcal{M}_X
\end{equation}
by the Serre subcategories $\mathcal{M}_X^i$ of those coherent sheaves whose 
support has codimension $\geq i$.

\begin{prop}
\label{prop:spectral seq}
\begin{enumerate}
\item 
{\rm(}Quillen spectral sequence {\rm \cite{Qui73}, \cite{Bal09}.)} 
For a noetherian scheme $X$, 
the filtration $\mathrm{(\ref{eq:filtration of MX})}$ 
induces the strongly convergent spectral sequence
\begin{equation}
\label{eq:Quillen spec seq}
E_1^{p,q}(X)=\underset{x\in X^p}{\bigoplus}K_{-p-q}(k(x))\Rightarrow G_{-p-q}(X),
\end{equation}
where $X^p$ is the set of points codimension $p$ in $X$. 
Moreover if $X$ is regular separated, then we have the canonical isomorphism
\begin{equation}
\label{eq:E2 term of Quillen spec seq}
E_2^{p,-p}(X)\simeq\operatorname{CH^p}(X).
\end{equation}

\item
{\rm(}Brown-Gersten-Thomason spectral sequence {\rm \cite{BG73}, \cite{TT90}.)} 
For a noetherian scheme of finite Krull dimension $X$, 
we have the canonical equivalence
\begin{equation}
\label{eq:K^B Zariski descent}
\mathbb{K}(X)\simeq \mathbb{H}_{\operatorname{Zar}}^{\bullet}(X;\mathbb{K}(-)).
\end{equation}
In particular, 
there exists the strongly convergent spectral sequence
\begin{equation}
{}'E_2^{p,q}(X)=\operatorname{H}_{\operatorname{Zar}}^p
(X;\widetilde{\mathbb{K}}_{q})
\Rightarrow \mathbb{K}_{q-p}(X)
\end{equation}
where $\widetilde{\mathbb{K}}_{q}$ is the Zariski sheafication of the presheaf 
$U\mapsto \pi_q\mathbb{K}(U)$ of non-connective $K$-theory. 
\end{enumerate}
\qed
\end{prop}

\noindent
To regard the isomorphisms $\mathrm{(\ref{eq:Bloch formula})}$ as 
the isomorphisms of $E_2$-terms of the spectral sequences above, 
Gersten gave the following observation in \cite{Ger73}. 
For simplicity, for a commutative noetherian ring $A$ with $1$, 
we write $\mathcal{M}_A$ and $\mathcal{M}^p_A$ for 
$\mathcal{M}_{\operatorname{Spec} A}$ and 
$\mathcal{M}_{\operatorname{Spec} A}^p$ respectively. 

\begin{prop}
\label{prop:Gersten's equiv}
The following conditions are equivalent:
\begin{enumerate}
\item
For any regular separated noetherian scheme $X$, we have the 
canonical isomorphism between $E_2$-terms of Quillen and Brown-Gersten-Thomason spectral sequences
$$E_2^{p,q}(X)\simeq{}'E_2^{p,-q}(X).$$
In particular the isomorphisms $\mathrm{(\ref{eq:Bloch formula})}$ hold for $X$ and any non-negative integer $n$.

\item
For any commutative regular local ring $R$ of dimension , 
$E_1$-terms of 
Quillen spectral sequence $\mathrm{(\ref{eq:Quillen spec seq})}$ 
for $\operatorname{Spec} R$ yields 
an exact sequence
\begin{equation}
0\to K_n(R) \to K_n(\operatorname{frac}(R))\to 
\underset{\operatorname{ht}\mathfrak{p}=1}{\bigoplus}
K_{n-1}(k(\mathfrak{p}))\to
\underset{\operatorname{ht}\mathfrak{p}=2}{\bigoplus}
K_{n-2}(k(\mathfrak{p}))\to\cdots.
\end{equation}

\item
For any commutative regular local ring $R$ and 
natural number $1\leq p\leq \dim R$, the canonical inclusion 
$\mathcal{M}^{p}_R \hookrightarrow \mathcal{M}_R^{p-1}$ induces the zero map on 
$K$-theory
$$K(\mathcal{M}_R^{p}) \to K(\mathcal{M}_R^{p-1})$$
where $K(\mathcal{M}_R^i)$ denotes the $K$-theory of 
the abelian category $\mathcal{M}_R^i$.
\end{enumerate}
\qed
\end{prop}

\noindent
Here is Gersten's conjecture:

\begin{con}[Gersten's conjecture]
\label{conj:Gersten's conjecture}
The conditions above are true for any commutative regular local ring.
\end{con}

\subsection*{Historical Note}

\noindent
The conjecture has been proved in the following cases:

\subsubsection*{The case of general dimension}

\begin{enumerate}
\item
If $A$ is of equi-characteristic, 
then Gersten's conjecture for $A$ is true.\\
We refer to \cite{Qui73} for special cases, 
and the general cases \cite{Pan03} can be deduced from limit 
argument and Popescu's general N\'eron desingularization 
\cite{Pop86}. 
(For the case of commutative discrete valuation rings, 
it was first proved by Sherman \cite{She78}).

\item
If $A$ is smooth over some commutative discrete valuation ring 
$S$ and satisfies some condition, 
then Gersten's conjecture for $A$ is true. \cite{Blo86}.

\item
If $A$ is smooth over some commutative discrete valuation ring $S$ and 
if we accept Gersten's conjecture for $S$, 
then Gersten's conjecture for $A$ is true. \cite{GL87}. 
See also \cite{RS90}. 
\end{enumerate}

\subsubsection*{The case of that $A$ is a commutative discrete valuation ring}

\begin{enumerate}
\item
For the cases of $n=0,\ 1$ are classical and 
$n=2$ was proved by Dennis and Stein, 
announced in \cite{DS72} and proved in \cite{DS75}.

\item
If its residue field is algebraic over a finite field, then Gersten's conjecture for $A$ is true.\\
We refer to \cite{Ger73} for the cases of 
that its residue field is a finite field, 
the proof of general cases \cite{She82} 
is improved from that of special cases 
by using Swan's result \cite{Swa63}, 
the universal property of algebraic $K$-theory \cite{Hil81} 
and limit argument.
\end{enumerate}

\subsubsection*{The case for Grothendieck groups}

Gersten's conjecture for Grothendieck groups has several equivalent forms. 
Namely we can show the following three conditions are equivalent. 
See \cite{CF68}, \cite{Lev85}, \cite{Dut93} and \cite{Dut95}.

\begin{enumerate}
\item
{\bf (Gersten's conjecture for Grothendieck groups).}\ 
{\it For any commutative regular local ring $R$ and 
natural number $1\leq p\leq \dim R$, the canonical inclusion 
$\mathcal{M}^{p}_R \hookrightarrow \mathcal{M}_R^{p-1}$ induces the zero map on 
Grothendieck groups
$$K_0(\mathcal{M}_R^{p}) \to K_0(\mathcal{M}_R^{p-1}).$$}
\item
{\bf (Generator conjecture).}\ 
{\it
For any commutative regular local ring $R$ and 
any natural number $0\leq p \leq \dim R$, 
the Grothendieck group 
$K_0(\mathcal{M}^p_R)$ is generated by 
cyclic modules $R/(f_1,\cdots,f_p)$ where 
the sequence 
$f_1,\cdots,f_p$ forms an $R$-regular sequence. 
}

\item
{\bf (Claborn and Fossum conjecture).}\ 
{\it For any commutative regular local ring $R$, 
the Chow homology group $\operatorname{CH}_k(\operatorname{Spec} R)$ is trivial for any $k<\dim R$.}
\end{enumerate}

\noindent
For historical notes of the generator conjecture, please see \cite{Moc13a}.

\subsubsection*{The case for singular varieties}

Gersten's conjecture for non-regular rings is 
in general false in the literal sense of the word and several 
appropriate modified versions are studied by many authors. 
See \cite{DHM85}, \cite{Smo87}, \cite{Lev88}, \cite{Bal09}, \cite{HM10} 
and \cite{Moc13a}. 
See also \cite{Mor15} and \cite{KM16}.

\subsubsection*{Other cohomology theories}

\begin{enumerate}
\item
For torsion coefficient $K$-theory, Gersten's conjecture for a commutative discrete valuation ring is true. \cite{Gil86}, \cite{GL00}.

\item
For Gersten's conjecture for Milnor $K$-theory, see \cite{Ker09}, \cite{Dah15}.

\item
For Gersten's conjecture for Witt groups, 
see \cite{Par82}, \cite{OP99}, \cite{BW02} and \cite{BGPW03}.

\item
For an analogue of Gersten's conjecture 
for bivariant $K$-theory, see \cite{Wal00}. 

\item
In the proof of geometric case of Gersten's conjecture by Quillen in \cite{Qui73}, 
he introduced a strengthning of Noether normalization theorem. 
There are several variants of Noether normalization theorem 
in \cite{Oja80}, \cite{GS88}, \cite{Gab94} and \cite{Wal98} 
and by utilizing them, 
there exists Gersten's conjecture type theorem 
for universal exactness (see \cite{Gra85}) for Cousin complexes 
(see \cite{Har66}) of certain cohomology theories. 
For axiomatic approaches of these topics, see \cite{BO74}, 
\cite{C-THK97}. See also \cite{Gab93}, \cite{C-T95}, \cite{Lev08} and \cite{Lev13}.

\item
For an analogue of Gersten's conjecture 
for Hochschild coniveau spectral sequences, see \cite{BW16}. 

\item 
For an analogue of Gersten's conjecture for infinitesimal theory, 
see \cite{DHY15}.

\item
For Gersten's complexes for homotopy invariant Zariski sheaves with 
transfers, see \cite{Voe00}, \cite[Lecture 24]{MVW06}. 
See also \cite{SZ03}. 
For injectivity result for pseudo pretheory, see \cite{FS02}.
\end{enumerate}

\subsubsection*{Logical connections with other conjectures}

\begin{enumerate}
\item
Some conjectures imply that Gersten's conjecture for a commutative discrete valuation ring is true. \cite{She89}.

\item
Parshin conjecture in \cite{Bei84} 
implies Gersten's conjecture of motivic cohomology 
for a localization of smooth varieties over a Dedekind ring. 
(see \cite{Gei04}.) 
It is also known that Tate-Beilinson conjecture (see \cite{Tat65}, 
\cite{Bei87} and \cite{Tat94}) 
implies Parshin conjecture. (see \cite{Gei98}.)
\end{enumerate}

\subsubsection*{Counterexample for non-commutative discrete valuation rings (Due to Kazuya Kato)}

\noindent
In Gersten's conjecture, 
the assumption of commutativity is essential. 
Let $D$ be a skew field finite over
$\mathbf{Q}_p$, $A$ its integer ring and $a$ its prime element. 
As the inner automorphism of $a$ 
induces non-trivial automorphism 
on its residue field,
we have $x \in A^{\times}$ with $y=axa^{-1}x^{-1}$ 
is non vanishing in its residue field, 
a fortiori in $K_1(A)={(A^{\times})}^{\operatorname{ab}}$. 
On the other hand $y$ is a 
commutator in $D^\times$. 
Hence it turns out that the canonical map 
$K_1(A) \to K_1(D)={(D^{\times})}^{\operatorname{ab}}$ 
is not injective.

\section{Idea of the proof}
\label{sec:idea of the proof}

My idea of how to prove Gersten's conjecture has come from weight argument of Adams operations 
in \cite{GS87} and \cite{GS99}. 
In my viewpoint, 
difficulty of solving Gersten's conjecture 
consists of ring theoretic side and 
homotopy theoretic side. 
We will explain ideas about how to overcome each difficulty.

\subsection*{Ring theoretic side}
\label{subsec:ring theoretic side}

Combining the results in \cite{GS87} and \cite{TT90}, 
for a commutative regular local ring $R$, 
there exists Adams operations $\{\varphi_k\}_{k\geq 0}$ 
on $K_0(\mathcal{M}_R^p)$ and we have the equality
\begin{equation}
\label{eq:Adams operation}
\varphi_k([R/(f_1,\cdots,f_p)]=k^p[R/(f_1,\cdots,f_p)]
\end{equation}
where a sequence $f_1,\cdots,f_p$ is an $R$-regular sequence. 
Thus roughly saying, 
the generator conjecture says that for each $p$, $K_0(\mathcal{M}^p_R)$ 
is spanned by objects of weight $p$ and Gersten's conjecture 
could follow from weight argument of Adams operations. 
We illustrate how to prove that the generator conjecture implies 
Gersten's conjecture for $K_0$ without using Adams operations.

\begin{proof}
Let a sequence $f_1,\cdots,f_p$ be an $R$-regular sequence. Then 
there exists the short exact sequence 
$$0\to R/(f_1,\cdots,f_{p-1})\overset{f_p}{\to}R/(f_1,\cdots,f_{p-1})\to R/(f_1,\cdots,f_p)\to 0$$
in $\mathcal{M}_R^{p-1}$. Thus 
the class $[R/(f_1,\cdots,f_p)]$ in $K_0(\mathcal{M}_R^p)$ goes to 
$$[R/(f_1,\cdots,f_p)]=[R/(f_1,\cdots,f_{p-1})]-[R/(f_1,\cdots,f_{p-1})]=0$$
in $K_0(\mathcal{M}_R^{p-1})$.
\end{proof}

\noindent
{\bf Strategy 1.}

\noindent
We will establish and prove a higher analogue of 
the generator conjecture.

\bigskip
\noindent
Inspired from the works \cite{Iwa59}, \cite{Ser59}, \cite{Bou64}, 
\cite{Die86} and \cite{Gra92}, 
we establish 
a classification theory of modules by utilizing cubes 
in \cite{Moc13a}, \cite{Moc13b} and \cite{MY14}. 
In this article, 
we will implicitly use these theory and simplify the arguments 
in \cite{Moc13a} and \cite{Moc16a}.

\subsection*{Homotopy theoretic side}
\label{subsec:homotopy theoretic side}

Roghly saying, 
we will try to compare the following two functors on $K$-theory.
We denote the category of bounded complexes on $\mathcal{M}_R^{p-1}$ 
by $\operatorname{Ch}_b(\mathcal{M}_R^{p-1})$.
$$\mathcal{M}_R^p \to \operatorname{Ch}_b(\mathcal{M}_R^{p-1}),$$
$$\ \ \ \ \ \ \ \ \ \ \ \ \ \ \ \ \ \ \ \ \ \ R/(f_1,\cdots,f_p)\mapsto
\begin{cases}
\begin{bmatrix}
R/(f_1,\cdots,f_p)\\
\downarrow f_p\\
R/(f_1,\cdots,f_p)
\end{bmatrix}
\underset{\text{qis}}{\sim} R/(f_1,\cdots,f_p)\\
\begin{bmatrix}
R/(f_1,\cdots,f_p)\\
\downarrow \operatorname{id}\\
R/(f_1,\cdots,f_p)
\end{bmatrix}
\underset{\text{qis}}{\sim} 0
\end{cases}$$
The functors above shall be homotopic to each other on 
$K$-theory by the additivity theorem. 
A problem is that 
the functors above are not {\bf $1$-funtorial}!! 
We need to a good notion of $K$-theory for 
higher category theory or 
need to discuss more subtle argument for 
such exotic functors.

\noindent
{\bf Strategy 2.}

\noindent
We give a modified definition of algebraic $K$-theory 
in a particular situation 
and establish a technique of rectifying lax functors 
to $1$-functors and 
by utilizing 
this definition and these techniques, 
we will treat such exotic functors inside the 
classical Waldhausen $K$-theory.

\section{Strategy of the proof}
\label{sec:Strategy of proof}

\noindent
Let $A$ be a commutative noetherian ring with $1$ and 
let $I$ be an ideal of $A$ with codimension $Y=V(I)\geq p$ in 
$\operatorname{Spec} A$. 
Let $\mathcal{M}_A^I$ be a full subcategory of 
$\mathcal{M}_A^p$ consisting of 
those modules $M$ supported on $Y=V(I)$ and 
$\mathcal{M}_{A,\operatorname{red}}^I$ be 
a full subcategory of $\mathcal{M}_A^I$ consisting of 
those modules $M$ such that $IM$ are trivial. 
We call a module in $\mathcal{M}_{A,\operatorname{red}}^I$ a 
{\it reduced module 
with respect to $I$}.
Let $\mathcal{P}_{A}$ be the category of finitely generated 
projective $A$-modules. 
For a commutative regular local ring $R$, 
let $J$ be an ideal generated by $R$-regular sequence 
$f_1,\cdots,f_p$ such that $f_i$ is an prime element for any $1\leq i\leq p$. 

\smallskip
\noindent
First notice the following isomorphisms:

\begin{equation}
\label{eq:Gersten strategy 1}
K(\mathcal{M}_R^p)\underset{\textbf{I}}{\overset{\scriptstyle{\sim}}{\to}} 
\underset{\substack{\operatorname{codim}_{\operatorname{Spec} R }V(I)= p \\
\operatorname{Spec} R/I \underset{\text{regular}}{\hookrightarrow}
\operatorname{Spec} R}}{\operatorname{colim}}K(\mathcal{M}_{R}^I),
\end{equation}
\begin{equation}
\label{eq:Gersten strategy 2}
K(\mathcal{P}_{R/J})
\underset{\textbf{II}}{\overset{\scriptstyle{\sim}}{\to}} 
K(\mathcal{M}_{R,\operatorname{red}}^J)
\underset{\textbf{III}}{\overset{\scriptstyle{\sim}}{\to}} K(\mathcal{M}_R^J).
\end{equation}

\noindent
Since $R$ is Cohen-Macaulay, 
the ordered set of all ideals of $R$ 
that contains an $R$-regular sequence of length $p$ 
with usual inclusion is directed. 
Thus $\mathcal{M}_R^p$ is the filtered limit 
$\underset{I}{\operatorname{colim}}\mathcal{M}_{R}^I $ 
where $I$ runs through any 
ideal generated by any $R$-regular sequence of length $p$. 
Thus the isomorphism \textbf{I} follows from cocontinuity of $K$-theory. 
The isomorphism \textbf{II} follows from 
the resolution theorem and regularity of $R$. 
Finally the isomorphism \textbf{III} follows from the d\'evissage theorem. 
Let $\mathfrak{f}_S=\{f_s\}_{s\in S}$ be an $R$-regular sequence such that 
$f_s$ is a prime element of $R$ for any $s\in S$. 
We call such a sequence $\mathfrak{f}_S$ a {\it prime regular sequence}. 
For a commutative noetherian ring $A$ with $1$ and for an ideal $I$ of $A$, 
let $\mathcal{M}_A^I(1)$ and $\mathcal{M}_{A,\operatorname{red}}^I(1)$ 
be the full subcategory of $\mathcal{M}_A^I$ and $\mathcal{M}_{A,\operatorname{red}}^I$ respectively 
consisting of 
those $A$-modules $M$ with $\operatorname{projdim}_A M\leq 1$. 
We fix an element $s\in S$. 
Since 
the inclusion functor 
$\mathcal{P}_{R/\mathfrak{f}_SR}\hookrightarrow \mathcal{M}_R^{\# S-1}$ 
factors through $\mathcal{M}_{R/\mathfrak{f}_{S\smallsetminus\{s\}}R}(1)$, 
the problem reduce to the following:

\noindent
{\it
For any prime regular sequence $\mathfrak{f}_{S}=\{f_s\}_{s\in S}$ of $R$ 
and any element $s\in S$, 
the inclusion functor $\mathcal{P}_{R/\mathfrak{f}_SR}\hookrightarrow 
\mathcal{M}_{R/\frak{f}_{S\smallsetminus\{s\}}R}(1)$ 
induces 
the zero map 
$$K(\mathcal{P}_{R/\mathfrak{f}_SR})\to 
K(\mathcal{M}_{R/\frak{f}_{S\smallsetminus\{s\}}R}(1))$$
on $K$-theory.}

\noindent
For simplicity we set $B=R/\frak{f}_{S\smallsetminus\{s\}}R$ and $g=f_s$. 
We let 
$\operatorname{Ch}_b(\mathcal{M}_B(1))$ denote 
the category of bounded complexes on $\mathcal{M}_B(1)$. 
(We use homological index notation.) 
Let $\mathcal{C}$ be the full subcategory of 
$\operatorname{Ch}_b(\mathcal{M}_B(1))$ 
consisting of those complexes $x$ such that 
$x_i=0$ unless $i=0$, $1$ and $x_1$ and $x_0$ are 
free $B$-modules and 
the bounded map $d^x\colon x_1\to x_0$ is injective and 
$\operatorname{H}_0 x:=\operatorname{Coker} (x_1\overset{d^x}{\to} x_0)$ 
is annihilated by $g$. 
We can show that 
$\mathcal{C}$ is an idempotent exact category such that 
the inclusion functor $\eta\colon\mathcal{C}\hookrightarrow 
\operatorname{Ch}_b(\mathcal{M}_B(1))$ is exact and 
reflects exactness and the functor 
$\operatorname{H}_0\colon \mathcal{C}\to \mathcal{P}_{B/gB}$ is exact. 
Thus we obtain the commutative diagram 
$${\xymatrix{
K(\mathcal{C}) \ar[r]^{\!\!\!\!\!\!\!\!\!\!\!\!\!\!\!\!\!\!\!\!\!\!\!\!\!\!\!\!\!\!\!\!K(\eta)} \ar[d]_{K(\operatorname{H}_0)} & 
K(\operatorname{Ch}_b(\mathcal{M}_B(1));\operatorname{qis})\\
K(\mathcal{P}_{B/gB}) \ar[r] & K(\mathcal{M}_B(1)) \ar[u]^{\wr}_{\textbf{I}}
}}$$
where $\operatorname{qis}$ is the class of all quasi-isomorphisms in 
$\operatorname{Ch}_b(\mathcal{M}_B(1))$ and 
the map \textbf{I} 
which is induced from the inclusion functor 
$\mathcal{M}_B(1)\hookrightarrow 
\operatorname{Ch}_b(\mathcal{M}_B(1))$ is 
a homotopy equivalence by Gillet-Waldhausen theorem. 
We will prove that

\begin{enumerate}
\item 
The map $K(\operatorname{H}_0)$ is a split epimorphism 
in the stable category of spectra 
(see \S \ref{sec:split epimorphism theorem}) 
and

\item
The map $K(\eta)$ is the zero map in the stable category of spectra. 
(See \S \ref{sec:zero map theorem})
\end{enumerate}

\noindent
Assertion $\mathrm{1.}$ corresponds with Strategy 1 and 
assertion $\mathrm{2.}$ corresponds with Strategy 2 
in the previous section.

\section{Split epimorphism theorem}
\label{sec:split epimorphism theorem}

\noindent
In this section, we will give a brief proof of assertion 
that $K(\operatorname{H}_0)$ is a split epimorphism 
in the stable category of spectra. 

\bigskip
\noindent
Let $\mathcal{D}$ be the full subcategory of 
$\operatorname{Ch}_b(\mathcal{M}_B(1))$ 
consisting 
of those complexes $x$ such that $x_i=0$ 
unless $i=0$, $1$ and 
the bounded map $d^x\colon x_1\to x_0$ is injective and 
$\operatorname{H}_0 x:=\operatorname{Coker} (x_1\overset{d^x}{\to} x_0)$ 
is in $\mathcal{M}_{B,\operatorname{red}}^{gB}(1)$. 
We can show that $\mathcal{D}$ is an exact category 
such that the inclusion functor 
$\mathcal{D}\hookrightarrow \operatorname{Ch}_b(\mathcal{M}_B(1))$ 
is exact and reflects exactness and we can also show that 
the functor $\operatorname{H}_0\colon \mathcal{D}\to \mathcal{M}_{B,\operatorname{red}}^{gB}(1) $ is exact. 
Thus we obtain the commutative square below
$$\xymatrix{
K(\mathcal{C}) \ar[r] \ar[d]_{K(\operatorname{H}_0)} & 
K(\mathcal{D}) \ar[d]^{K(\operatorname{H}_0)}\\
K(\mathcal{P}_{B/gB}) \ar[r] & K(\mathcal{M}_{B,\operatorname{red}}^{gB}(1)).
}$$
where the horizontal maps are induced from the inclusion functors. 
Since the functor $\operatorname{H}_0\colon \mathcal{D}\to \mathcal{M}_{B,\operatorname{red}}^{gB}(1)$ admits a section which is defined by 
sending an object $x$ in 
$\mathcal{M}_{B,\operatorname{red}}^{gB}(1)$ to 
the complex $[0\to x]$ in $\mathcal{D}$, 
the right vertical map in the diagram above is a split epimorphism. 
Moreover the inclusion functors $\mathcal{P}_{B/gB}\hookrightarrow \mathcal{M}_{B,\operatorname{red}}^{gB}(1)$ and $\mathcal{C}\hookrightarrow\mathcal{D}$ 
induce equivalences of 
triangulated categories on bounded derived categories respectively. 
(Compare \cite[2.21]{Moc13a} and \cite[2.1.1]{Moc16a}.) 
Thus the horizontal maps in the diagram above are homotopy equivalences 
and the left vertical map is also a split epimorphism in the stable category 
of spectra. \qed

\section{Zero map theorem}
\label{sec:zero map theorem}

\noindent
In this section we will give an outline of the proof 
of assertion that $K(\eta)$ is the zero map in the stable category of spectra.

\noindent
Let $\mathcal{B}$ be the full subcategory of 
$\operatorname{Ch}_b\mathcal{M}_{B}(1)$ consiting 
of those complexes $x$ such that $x_i=0$ unless $i=0$ or $i=1$. 
Let $\mathfrak{s}_i\colon \mathcal{B} \to \mathcal{M}_B(1)$ 
($i=0$, $1$) be an exact functor 
defined by sending an object $x$ in 
$\mathcal{B}$ to $x_i$ in $\mathcal{M}_B(1)$. 
By the additivity theorem, the map 
$\mathfrak{s}_1\times \mathfrak{s}_2\colon 
iS_{\cdot}\mathcal{B}\to 
iS_{\cdot}\mathcal{M}_B(1)\times iS_{\cdot}\mathcal{M}_B(1)$ 
is a homotopy equivalence. 
Let $j\colon \mathcal{B} \to \operatorname{Ch}_b(\mathcal{M}_{B}(1))$ 
be the inculsion functor. 

\noindent
We wish to define two exact 'functors' 
$\mu_1,\ \mu_2\colon \mathcal{C} \to \mathcal{B}$ 
which satisfy the following conditions:

\begin{enumerate}
\item
We have the equality 
\begin{equation}
\label{eq:mu1=mu2}
\mathfrak{s}_1\times\mathfrak{s}_2\mu_1=\mathfrak{s}_1\times\mathfrak{s}_2\mu_2.
\end{equation}

\item
There are natural transformations $\eta \to j\mu_1$ and 
$0 \to j\mu_2$ such that all componets are quasi-isomorphisms.
\end{enumerate}

\noindent
Then we have the equalities
$$
K(\eta)=K(j\mu_1)=K(j)K(\mathfrak{s}_1\times\mathfrak{s}_2)^{-1}K(\mathfrak{s}_1\times\mathfrak{s}_2\mu_1)=
K(j)K(\mathfrak{s}_1\times\mathfrak{s}_2)^{-1}K(\mathfrak{s}_1\times\mathfrak{s}_2\mu_2)=K(j\mu_2)=0.
$$

\noindent
To define the 'functors' $\mu_i$ ($i=1$, $2$), 
we analyze morphisms in $\mathcal{C}$. 

\subsection{Structure of $\mathcal{C}$}
\label{sec:structure of C}

\begin{df}
\label{df:(n,m)}
(Compare \cite[1.1.3.]{Moc16a}.)\ \ 
For a pair of non-negative integers $(n,m)$, 
we write ${(n,m)}_B$ for the complex of the form 
$\displaystyle{
\begin{bmatrix}
B^{\oplus n}\oplus B^{\oplus m}\ \ \ \ \ \ \ \ \ \ \  \\
\ \ \ \ \ \ \ \ \ \downarrow\begin{pmatrix}gE_n & 0\\ 0 & E_m \end{pmatrix}\\
B^{\oplus n}\oplus B^{\oplus m}\ \ \ \ \ \ \ \ \ \ \ 
\end{bmatrix}
}$
in $\mathcal{C}$ where $E_k$ is the $k\times k$ unit matrix. 
\end{df}

\begin{lem}
\label{lem:structure of C}
$ $

\noindent
\begin{enumerate}
\renewcommand{\labelenumi}{$\mathrm{(\arabic{enumi})}$}
\item
{\rm(}Compare {\rm \cite[1.2.10., 1.2.13.]{Moc16a}.)}\ \ 
Let $n$ be a positive integer. 
For any endomorphism $a\colon {(n,0)}_B\to{(n,0)}_B$, 
the following conditions are equivalent.

\begin{enumerate}
\renewcommand{\labelenumii}{$\mathrm{(\roman{enumii})}$}
\item
$a$ is an isomorphism.

\item
$a$ is a quasi-isomorphism.
\end{enumerate}

\item
{\rm(}Compare {\rm\cite[2.17.]{Moc13a}.)}\ \ 
An object in $\mathcal{C}$ is projective. 
In particular, $\mathcal{C}$ is a semi-simple exact category.

\item
{\rm(}Compare {\rm\cite[1.2.15.]{Moc16a}.)}\ \ 
For any object $x$ in $\mathcal{C}$, 
there exists a pair of non-negative integers 
$(n,m)$ such that $x$ is isomorphic to ${(n,m)}_B$.
\end{enumerate}
\end{lem}

\begin{proof}
$\mathrm{(1)}$ 
We assume condition $\mathrm{(ii)}$. 
In the commutative diagram below
$$
\xymatrix{
B^{\oplus n } \ar[r]^{gE_n} \ar[d]_{a_1} & B^{\oplus n } \ar[r] \ar[d]_{a_0} & 
\operatorname{H}_0({(n,0)}_B) \ar[d]^{\operatorname{H}_0(a)}\\
B^{\oplus n } \ar[r]_{gE_n} & B^{\oplus n } \ar[r] & 
\operatorname{H}_0({(n,0)}_B),
}$$
first we will prove that $a_0$ is an isomorphism. 
Then $a_1$ is also an isomorphism by the five lemma. 
By taking determinant of $a_0$, 
we shall assume that $n=1$. 
Then assertion follows from Nakayama's lemma.

\smallskip
\noindent
$\mathrm{(2)}$ 
Let $t\colon y\twoheadrightarrow z$ 
be an admissible epimorphism in $\mathcal{C}$ 
and let $f\colon x\to z$ be a morphism in $\mathcal{C}$. 
Then since $\operatorname{H}_0(x)$ is a projective $B/gB$-module, 
there exist a homomorphism of $B/gB$-modules  
$\sigma\colon \operatorname{H}_0(x)\to \operatorname{H}_0(y)$ 
such that $\operatorname{H}_0 (t)\sigma=\operatorname{H}_0 (f)$. 
Since $x$ is a complex of free $B$-modules, 
there exists a morphism of complexes 
$s'\colon x\to y$ such that 
$\operatorname{H}_0(s')=\sigma$ and $ts'$ is chain homotopic 
to $f$ by \cite[Comparison theorem 2.2.6.]{Wei94}. 
Namely there is a map $h\colon x_0 \to z_1$ such that 
${(f-ts')}_0=d^zh$ and ${(f-ts')}_1=hd^x$. 
$$\xymatrix{
x_1 \ar[d]_{d^x} \ar[r]^{{(f-ts')}_1}& z_1 \ar[d]^{d^z}\\
x_0 \ar[ru]_h \ar[r]_{{(f-ts')}_0} & z_0. 
}$$
Since $x_0$ is projective, 
there is a map $u\colon x_0\to y_1$ such that $t_1u=h$. 
We set $s_1:={s'}_1+ud^x$ and $s_0:={s'}_0+d^yu$. 
Then we can check that $s$ is a morphism of complexes of $B$-modules 
and $f=ts$.

\smallskip
\noindent
$\mathrm{(3)}$ 
By considering $\displaystyle{x\otimes_B B\left [\frac{1}{g}\right]}$, 
we notice that $x_1$ and $x_0$ are same rank. 
Thus we shall assume $x_1=x_0=B^{\oplus m}$. 

\smallskip
\noindent
First we assume that $x$ is acyclic. 
Then the boundary map $d^x\colon x_1\to x_0$ 
is invertible and 
$$
\begin{bmatrix}
x_1\ \ \   \\
\ \downarrow d^x\\
x_0\ \ \ 
\end{bmatrix}
\begin{matrix}
\overset{d^x}{\to}
\\
\underset{\operatorname{id}}{\to}
\end{matrix}
\begin{bmatrix}
B^{\oplus m}\ \ \ \\
\ \ \ \downarrow\operatorname{id}_{B^{\oplus m}}\\
B^{\oplus m}\ \ \ 
\end{bmatrix}
$$
gives an isomorphism between $x$ and ${(0,m)}_B$. 
Thus we obtain the result in this case. 

\smallskip
\noindent
Next assume that $\operatorname{H}_0(x)\neq 0$. 
Then sicne $\operatorname{H}_0(x)$ is 
a finitely generated projective $B/gB$-module, 
there is a positive integer $n$ and an isomorphism 
$\sigma\colon {(B/gB)}^{\oplus n}\overset{\scriptstyle{\sim}}{\to} 
\operatorname{H}_0(x)$. 
Then by \cite[Comparison theorem 2.2.6.]{Wei94}, 
there exists morphisms of complexes in $\mathcal{C}$, 
${(n,0)}_B\overset{a}{\to} x$ and $x\overset{b}{\to}{(n,0)}_B$ 
such that $\operatorname{H}_0(a)=\sigma$ and 
$\operatorname{H}_0(b)=\sigma^{-1}$. 
Thus by $\mathrm{(1)}$, 
$ba$ is an isomorphism. 
By replacing $a$ with $a{(ba)}^{-1}$, 
we shall assume that $ba=\operatorname{id}$. 
Hence there exists a complex $y$ in $\mathcal{C}$ and a split 
exact sequence:
\begin{equation}
\label{eq:split exact seq}
{(n,0)}_B\overset{a}{\rightarrowtail} x \twoheadrightarrow y. 
\end{equation}
Since boundary maps of objects in $\mathcal{C}$ are injective, 
the functor $\operatorname{H}_0$ from $\mathcal{C}$ 
to the category of finitly generated projective 
$B/gB$-modules is exact. 
By taking $\operatorname{H}_0$ 
to the sequence $\mathrm{(\ref{eq:split exact seq})}$, 
it turns out that 
$y$ is acyclic and by the first paragraph, 
we shall assume that $y$ is of the form ${(0,m)}_B$. 
Hence $x$ is isomorphic to ${(n,m)}_B$.
\end{proof}

\begin{rem}[Morphisms of $\mathcal{C}$]
\label{rem:mor of C}
(Compare \cite[1.2.16.]{Moc16a}.)\ \ 
We can denote a morphism $\varphi\colon {(n,m)}_B\to {(n',m')}_B$ of 
$\mathcal{C}$ by 
$$
\begin{bmatrix}
B^{\oplus n}\oplus B^{\oplus m}\ \ \ \ \ \ \ \ \ \ \  \\
\ \ \ \ \ \ \ \ \ \downarrow\begin{pmatrix}gE_n & 0\\ 0 & E_m \end{pmatrix}\\
B^{\oplus n}\oplus B^{\oplus m}\ \ \ \ \ \ \ \ \ \ \ 
\end{bmatrix}
\begin{matrix}
\overset{\varphi_{1}}{\to}
\\
\underset{\varphi_{0}}{\to}
\end{matrix}
\begin{bmatrix}
B^{\oplus n'}\oplus B^{\oplus m'}\ \ \ \ \ \ \ \ \ \ \ \\
\ \ \ \ \ \ \ \ \ \downarrow\begin{pmatrix}gE_{n'} & 0\\ 0 & E_{m'} \end{pmatrix}\\
B^{\oplus n'}\oplus B^{\oplus m'}\ \ \ \ \ \ \ \ \ \ \ 
\end{bmatrix}
$$
with 
$\displaystyle{\varphi_{1}=
\begin{pmatrix}
\varphi_{(n',n)} & \varphi_{(n',m)}\\ 
g\varphi_{(m',n)} & \varphi_{(m',m)} 
\end{pmatrix}}$ and 
$\displaystyle{\varphi_{0}=
\begin{pmatrix} 
\varphi_{(n',n)} & g\varphi_{(n',m)}\\ 
\varphi_{(m',n)} & \varphi_{(m',m)}
\end{pmatrix}}$ 
where $\varphi_{(i,j)}$ are $i\times j$ matrices whose coefficients are in $B$. 
In this case we write 
\begin{equation}
\label{eq:matrix presentation}
\displaystyle{
\begin{pmatrix} 
\varphi_{(n',n)} & \varphi_{(n',m)}\\ 
\varphi_{(m',n)} & \varphi_{(m',m)}
\end{pmatrix}}
\end{equation}
for $\varphi$. 
In this matrix presentation of morphisms, the composition 
of morphisms between objects 
${(n,m)}_B\overset{\varphi}{\to}{(n',m')}_B\overset{\psi}{\to}{(n'',m'')}_B$ 
in $\mathcal{C}$ is described by 
\begin{multline}
\label{eq:composition}
\begin{pmatrix} 
\psi_{(n'',n')} & \psi_{(n'',m')}\\ 
\psi_{(m'',n')} & \psi_{(m'',m')}
\end{pmatrix}
\begin{pmatrix}
\varphi_{(n',n)} & \varphi_{(n',m)}\\ 
\varphi_{(m',n)} & \varphi_{(m',m)}
\end{pmatrix}\\
=
\begin{pmatrix}
\psi_{(n'',n')}\varphi_{(n',n)}+g\psi_{(n'',m')}\varphi_{(m',n)} & 
\psi_{(n'',n')}\varphi_{(n',m)}+\psi_{(n'',m')}\varphi_{(m',m)}\\
\psi_{(m'',n')}\varphi_{(n',n)}+\psi_{(m'',m')}\varphi_{(m',n)} & 
g\psi_{(m'',n')}\varphi_{(n',m)}+\psi_{(m'',m')}\varphi_{(m',m)}
\end{pmatrix}.
\end{multline}
Thus the category 
$\mathcal{C}$ is categorical equivalent to the category 
whose objects are oredered pair of 
non-negative integers $(n,m)$ and whose morphisms 
from an object $(n,m)$ to $(n',m')$ are $2\times 2$ matrices 
of the form 
(\ref{eq:matrix presentation}) 
of $i\times j$ matrices $\varphi_{(i,j)}$ whose coefficients are in $B$ 
and compositions are given by the formula (\ref{eq:composition}). 
We sometimes identify these two categories.  
\end{rem}

\subsection{Modified algebraic $K$-theory}

\noindent
A candidate of a pair of 
$\mu_i\colon \mathcal{C}\to \mathcal{B}$ ($i=1$, $2$) are following. 
We define ${\mu'}_1,\ {\mu'}_2\colon\mathcal{C}\to\mathcal{B}$ 
to be associations by sending an object 
${(n,m)}_B$ in $\mathcal{C}$ 
to ${(n,0)}_B$ and ${(0,n)}_B$ 
respecitvely and a morphism 
$$\displaystyle{\varphi=\begin{pmatrix} 
\varphi_{(n',n)} & \varphi_{(n',m)}\\
\varphi_{(m',n)} & \varphi_{(m',m)}
\end{pmatrix}\colon {(n,m)}_B\to {(n',m')}_B}$$ in $\mathcal{C}$ 
to 
$\displaystyle{\begin{bmatrix}
B^{\oplus n}\\
\downarrow gE_n\\
B^{\oplus n}
\end{bmatrix}
\begin{matrix}
\overset{\varphi_{(n',n)}}{\to}
\\
\underset{\varphi_{(n',n)}}{\to}
\end{matrix}
\begin{bmatrix}
B^{\oplus n'}\\
\downarrow gE_{n'}\\
B^{\oplus n'}
\end{bmatrix}}$ and 
$\displaystyle{\begin{bmatrix}
B^{\oplus n}\\
\downarrow E_n\\
B^{\oplus n}
\end{bmatrix}
\begin{matrix}
\overset{\varphi_{(n',n)}}{\to}
\\
\underset{\varphi_{(n',n)}}{\to}
\end{matrix}
\begin{bmatrix}
B^{\oplus n'}\\
\downarrow E_{n'}\\
B^{\oplus n'}
\end{bmatrix}}$ 
respectively. 
Notice that they are not $1$-functors. 
We need to make revisions in the previous idea. 
We introduce a modified version of algebraic $K$-theory of $\mathcal{C}$. 

\begin{df}[Triangular morphisms]
\label{df:triangular morphism}
(Compare \cite[2.3.5.]{Moc16a}.)\ \ 
We say that a morphism 
$\varphi\colon {(n,m)}_B\to {(n',m')}_B$ 
in $\mathcal{C}$ of the form $\mathrm{(\ref{eq:matrix presentation})}$ 
is an {\it upper triangular} if 
$\varphi_{(m',n)}$ is the zero morphism, 
and say that $\varphi$ is a {\it lower triangular} if 
$\varphi_{(n',m)}$ is the zero morphism. 
We denote the class of all upper triangular isomorphisms in $\mathcal{C}$
by $i^{\bigtriangleup}$. 
Next we define $S^{\bigtriangledown}_{\cdot}\mathcal{C}$ to be a simplicial 
subcategory of $S_{\cdot}\mathcal{C}$ consisting of those objects 
$x$ such that $x(i\leq j) \to x(i'\leq j')$ is a lower triangular 
morphism for each $i\leq i', j \leq j'$. 
\end{df}

\begin{prop}
$ $

\noindent
\begin{enumerate}
\renewcommand{\labelenumi}{$\mathrm{(\arabic{enumi})}$}
\item 
{\rm(}Compare {\rm\cite[2.3.5]{Moc16a}.)}\ \ 
The inclusion functor 
$iS_{\cdot}^{\bigtriangledown}\mathcal{C} \to iS_{\cdot}\mathcal{C}$ 
is a homotopy equivalence.

\item
{\rm(}Compare {\rm\cite[2.3.6]{Moc16a}.)}\ \ 
The inclusion functor 
$i^{\bigtriangleup}S_{\cdot}^{\bigtriangledown}\mathcal{C} \to iS_{\cdot}^{\bigtriangledown}\mathcal{C}$ 
is a homotopy equivalence.

\item
{\rm(}Compare {\rm\cite[2.3.7]{Moc16a}.)}\ \ 
The associations $\mu'_1,\ \mu'_2\colon \mathcal{C}\to 
\operatorname{Ch}_b(\mathcal{M}_B(1))$ induces 
simplicial maps $\mu_1,\ \mu_2\colon i^{\bigtriangleup}S^{\bigtriangledown}_{\cdot}\mathcal{C}\to iS_{\cdot}\mathcal{B}$ 
respectively.

\item
{\rm(}Compare {\rm\cite[2.3.7]{Moc16a}.)}\ \ 
$\mu_1$ is homotopic to $\mu_2$.
\end{enumerate}
\end{prop}

\begin{proof}
$\mathrm{(1)}$ 
Since $\mathcal{C}$ is semi-simple by Lemma~\ref{lem:structure of C} 
$\mathrm{(2)}$, 
the inclusion functor 
$k\colon iS_{\cdot}^{\bigtriangledown}\mathcal{C}\to iS_{\cdot}\mathcal{C}$ 
is an equivalence of categories for each degree. 
Therefore the inclusion functor $k$ induces a weak homotopy equivalence 
$NiS_{\cdot}^{\bigtriangledown}\mathcal{C} \to NiS_{\cdot}\mathcal{C}$.

\smallskip
\noindent
$\mathrm{(2)}$ 
First for non-negative integer $n$, 
let $i_n\mathcal{C}$ be the full subcategory of $\mathcal{C}^{[n]}$ 
the functor category from the totally ordered set 
$[n]=\{0<1<\cdots<n \}$ to $\mathcal{C}$ 
consisting 
of those objects $x\colon [n] \to \mathcal{C}$ such that 
$x(i\leq i+1)$ is an isomorphism in $\mathcal{C}$ for any $0\leq i\leq n-1$. 
Next for integers $n\geq 1$ and $n-1\geq k\geq 0$, 
let $i_n\mathcal{C}^{(k)}$ be the full subcategory of 
$i_n\mathcal{C}$ consisting 
of those objects $x\colon [n] \to \mathcal{C}$ such that 
$x(i\leq i+1)$ is in $i^{\bigtriangleup}$ for any $k\leq i\leq n-1$. 
In particular $i_n\mathcal{C}^{(0)}=i_n^{\bigtriangleup}\mathcal{C}$ and 
by convention, we set $i_n\mathcal{C}^{(n)}=i_n\mathcal{C}$. 
There is a sequence of inclusion functors; 
$$i_n^{\bigtriangleup}\mathcal{C}=i_n\mathcal{C}^{(0)}\overset{j_0}{\hookrightarrow}
i_n\mathcal{C}^{(1)}\overset{j_1}{\hookrightarrow}\cdots \overset{j_n-1}{\hookrightarrow}i_n\mathcal{C}^{(n)}=i_n\mathcal{C}.$$
For each $0\leq k\leq n-1$, 
we will define $q_k\colon i_n\mathcal{C}^{(k+1)}\to i_n\mathcal{C}^{(k)}$ 
to be an exact functor as follows. 
First for any object $z$ in $i_n\mathcal{C}^{(k+1)}$, we shall assume that 
all $z(i)$ are the same object, namely 
$z(0)=z(1)=\cdots=z(n)$. 
Then we define $\alpha_z$ to be an isomorphism of $z(i)$ in $\mathcal{C}$ 
by setting $\alpha_z:=\operatorname{UT}(z(k\leq k+1))$. 
Here for the definition of 
the upper triangulation $\operatorname{UT}$ of $z(k\leq k+1)$, 
see Definition~\ref{df:uppertriangulation} below. 
Next for an object 
$x\colon[n]\to \mathcal{C}$ and a morphism $x\overset{\theta}{\to} y$ in $i_n\mathcal{C}^{(k+1)}$, 
we define $q_k(x)\colon [n]\to\mathcal{C} $ and 
$q_k(\theta)\colon q_k(x)\to q_k(y)$ 
to be an object and a morphism in $i_n\mathcal{C}^{(k)}$ respectively 
by setting 
\begin{equation}
q_k(x)(i):=x(i),
\end{equation}
\begin{equation}
q_k(x)(i\leq i+1):=
\begin{cases}
\alpha^{-1}_x x(k-1\leq k) & \text{if $i=k-1$}\\
x(k\leq k+1)\alpha_x & \text{if $i=k$}\\
x(i\leq i+1) &\text{otherwise},
\end{cases}
\end{equation}
\begin{equation}
q_k(\theta)(i):=
\begin{cases}
\alpha_y^{-1}\theta(k)\alpha_x & \text{if $i=k$}\\
\theta(i) & \text{otherwise}.
\end{cases}
\end{equation}
Obviously $q_kj_k=\operatorname{id}$. 
We define $\gamma^k\colon j_kq_k\overset{\scriptstyle{\sim}}{\to} \operatorname{id}$ to be a natural equivalence 
by setting for any object $x$ in $i_n\mathcal{C}^{(k+1)}$
\begin{equation}
\gamma^k(x)(i):=
\begin{cases}
\alpha_x & \text{if $i=k$}\\
\operatorname{id}_{x_i} & \text{otherwise}.
\end{cases}
\end{equation} 
Let $s_{\cdot}^{\bigtriangledown}:=\operatorname{Ob} S^{\bigtriangledown}$ be a variant 
of $s=\operatorname{Ob} S$-construction. 
Notice that there is a natural identification 
$s_{\cdot}^{\bigtriangledown}i_n\mathcal{C}^{(l)}=i_nS^{\bigtriangledown}\mathcal{C}^{(l)}$ 
for any $0\leq l\leq n$. 
We will show that $\gamma$ induces a simplicial homotopy between the maps 
$s_{\cdot}^{\bigtriangledown}j_kq_k$ and $s_{\cdot}^{\bigtriangledown}\operatorname{id}$. 
The proof of this fact is similar to \cite[Lemma 1.4.1]{Wal85}. 
The key point of well-definedness of the simplicial homotopy 
is that each component of $\gamma$ is lower triangular. 
Therefore the inclusion 
$i_nS_{\cdot}^{\bigtriangledown}\mathcal{C}^{(k)} \to 
i_nS_{\cdot}^{\bigtriangledown}\mathcal{C}^{(k+1)}$ 
is a homotopy equivalence. 
Hence by realization lemma \cite[Appendix A]{Seg74} or \cite[5.1]{Wal78}, 
$NiS_{\cdot}^{\bigtriangledown}\mathcal{C}^{(k)}\to 
NiS_{\cdot}^{\bigtriangledown}\mathcal{C}^{(k+1)}$ 
is also a homotopy equivalence 
for any $0\leq k\leq n-1$. 
Thus we complete the proof.

\smallskip
\noindent
$\mathrm{(3)}$ 
Notice that for a pair of 
composable morphisms in $\mathcal{C}$, 
\begin{equation}
\label{eq:comp mor in C}
{(n,m)}_B\overset{\varphi}{\to} {(n',m')}_B\overset{\psi}{\to} {(n'',m'')}_B,
\end{equation}

\begin{enumerate}
\renewcommand{\labelenumi}{$\mathrm{(\roman{enumi})}$}
\item
if both $\varphi $ and $\psi $ are upper triangular or 
both $\varphi $ and $\psi $ are 
lower triangular, then we have the equality 
${\mu'}_i(\psi\varphi)={\mu'}_i(\psi){\mu'}_i(\varphi)$ 
for $i=1$, $2$, 

\item
if the sequence $\mathrm{(\ref{eq:comp mor in C})}$ 
is exact in $\mathcal{C}$, then the sequence 
$${\mu'}_i({(n,m)}_B)\overset{{\mu'}_i(\varphi)}{\to} 
{\mu'}_i({(n',m')}_B)\overset{{\mu'}_i(\psi)}{\to} 
{\mu'}_i({(n'',m'')}_B) $$
is exact in $\mathcal{B}$ for $i=1$, $2$ 
by Lemma~\ref{lem:exact sequences of C} below and

\item
if $\varphi$ is an isomorphism in $\mathcal{C}$, then 
${\mu'}_i(\varphi)$ is an isomorphism in $\mathcal{B}$ for $i=0$, $1$ 
by Lemma~\ref{lem:11 and 22 are isomorphisms} below. 
\end{enumerate}

\bigskip
\noindent
Thus the associations ${\mu'}_1$ and ${\mu'}_2$ 
induce the simplicial functors 
$\mu_1,\mu_2\colon 
i^{\bigtriangleup}S_{\cdot}^{\bigtriangledown}\mathcal{C}\to iS_{\cdot}\mathcal{B}$.

\smallskip
\noindent
$\mathrm{(4)}$ 
Inspection shows the equalitiy 
$\mathrm{(\ref{eq:mu1=mu2})}$.  
Hence $\mu_1$ is homotopic to $\mu_2$ by the additivity theorem. 
\end{proof}

\begin{df}[Upside-down involution]
\label{df:UDI}
(Compare \cite[1.2.17.]{Moc16a}.)\ \ 
Let $x=[x_1\overset{d^x}{\to}x_0]$ be an object in 
$\mathcal{C}$. 
Since $d^x$ is a monomorphism and $x_1$ and $x_0$ have the same rank, 
$d^x$ is invertible in $\displaystyle{B\left [\frac{1}{g} \right ]}$ 
and $g{(d^x)}^{-1}\colon x_0\to x_1$ is a morphism of $B$-modules. 
We define $\operatorname{UD}\colon \mathcal{C}\to \mathcal{C}$ to be a functor 
by sending an object $[x_1\overset{d^x}{\to}x_0]$ to 
$[x_0\overset{g{(d^x)}^{-1}}{\to}x_1]$ 
and a morphism $\varphi\colon x\to y$ to
$$\begin{bmatrix}
x_0\ \ \ \ \ \ \ \ \ \ \\
\downarrow g{(d^{x})}^{-1}\\
x_1\ \ \ \ \ \ \ \ \ \
\end{bmatrix}
\begin{matrix}
\overset{\varphi_0}{\to}\\
\\
\underset{\varphi_1}{\to}
\end{matrix}
\begin{bmatrix}
y_0\ \ \ \ \ \ \ \ \ \ \\
\downarrow g{(d^y)}^{-1}\\
y_1\ \ \ \ \ \ \ \ \ \
\end{bmatrix}.$$
Namely we have the equations:
$$\operatorname{UD}({(n,m)}_B)={(m,n)}_B,$$
$$\operatorname{UD}\left (\begin{pmatrix}\varphi_{(n',n)}&\varphi_{(n',m)}\\
\varphi_{(m',n)} & \varphi_{(m',m)} \end{pmatrix}
\colon {(n,m)}_B\to{(n',m')}_B
 \right )=
\begin{pmatrix}\varphi_{(m',m)}&\varphi_{(m',n)}\\
\varphi_{(n',m)} & \varphi_{(n',n)} \end{pmatrix}.$$
Obviously $\operatorname{UD}$ is an involution and an exact functor.
We call $\operatorname{UD}$ the {\it upside-down involution}. 
\end{df}

\begin{lem}
\label{lem:11 and 22 are isomorphisms}
{\rm(}Compare {\rm\cite[1.2.18.]{Moc16a}.)}\ \ 
Let 
$\displaystyle{\varphi=\begin{pmatrix}
\varphi_{(n,n)} & \varphi_{(n,m)}\\
\varphi_{(m,n)} & \varphi_{(m,m)}
\end{pmatrix}\colon {(n,m)}_B\to {(n,m)}_B}$ be an isomorphism in $\mathcal{C}$. 
Then $\varphi_{(n,n)}$ and $\varphi_{(m,m)}$ are invertible. 
\end{lem}

\begin{proof}
For $\varphi_{(n,n)}$, assertion follows from Lemma~\ref{lem:structure of C} $\mathrm{(1)}$. 
For $\varphi_{(m,m)}$, we apply the same lemma to $\operatorname{UD}(\varphi)$. 
\end{proof}

\begin{df}[Upper triangulation]
\label{df:uppertriangulation}
(Compare \cite[2.2.5.]{Moc16a}.)\ \ 
Let 
$$\displaystyle{\varphi=\begin{pmatrix}
\varphi_{(n,n)} & \varphi_{(n,m)}\\
\varphi_{(m,n)} & \varphi_{(m,m)}
\end{pmatrix}\colon {(n,m)}_B\to {(n,m)}_B}$$
be an isomorphism in $\mathcal{C}$. 
By Lemma~\ref{lem:11 and 22 are isomorphisms}, 
$\varphi_{(m,m)}$ is an isomorphism. 
We define $\operatorname{UT}(\varphi)\colon {(n,m)}_B\to {(n,m)}_B$ 
to be a lower triangular isomorphism by the formula 
$\displaystyle{\operatorname{UT}(\varphi):=\begin{pmatrix} 
E_n & 0\\
-\varphi_{(m,m)}^{-1}\varphi_{(m,n)} & E_m
\end{pmatrix}}$. 
Then we have an equality
\begin{equation}
\label{eq:triangulation}
\varphi\operatorname{UT}(\varphi)=\begin{pmatrix}
\varphi_{(n,n)}-g\varphi_{(n,m)}\varphi_{(m,m)}^{-1}\varphi_{(m,n)} & \varphi_{(n,m)}\\
0 & \varphi_{(m,m)}
\end{pmatrix}.
\end{equation}
We call $\operatorname{UT}(\varphi)$ the {\it upper triangulation} of $\varphi$. 
Notice that if $\varphi$ is upper triangular, 
then $\operatorname{UT}(\varphi)=\operatorname{id}_{{(n,m)}_B}$. 
\end{df}

\begin{lem}[Exact sequences in $\mathcal{C}$]
\label{lem:exact sequences of C}
{\rm(}Compare {\rm\cite[1.2.19.]{Moc16a}.)}\ \ 
Let 
\begin{equation}
\label{eq:seq of Typ}
{(n,0)}_B \overset{\alpha}{\to} {(n',0)}_B \overset{\beta}{\to}
{(n'',0)}_B
\end{equation}
be a sequence of morphisms in $\mathcal{C}$ such that $\beta\alpha=0$. 
If the induced sequence of projective $B/gB$-modules
\begin{equation}
\label{eq:ind exact seq}
\operatorname{H}_0({(n,0)}_B) \overset{\operatorname{H}_0(\alpha)}{\to}
\operatorname{H}_0({(n',0)}_B) \overset{\operatorname{H}_0(\beta)}{\to}
\operatorname{H}_0({(n'',0)}_B)
\end{equation}
is exact, then the sequence $\mathrm{(\ref{eq:seq of Typ})}$ 
is also {\rm (}split{\rm )} exact.
\end{lem}

\begin{proof}
Since the sequence $\mathrm{(\ref{eq:ind exact seq})}$ 
is an exact sequence of projective 
$B/gB$-modules, it is a split exact sequence and hence 
we have the equality 
$n'=n+n''$ and 
there exists a homomorphism of $B/gB$-modules  
$\overline{\gamma}\colon\operatorname{H}_0({(n'',0)}_B)\to
\operatorname{H}_0({(n',0)}_B)$
such that 
$\operatorname{H}_0(\beta)\overline{\gamma}=\operatorname{id}_{\operatorname{H}_0({(n'',0)}_B)}$. 
Then by \cite[Comparison theorem 2.2.6.]{Wei94}, 
there is a morphism of complexes of $B$-modules 
$\gamma\colon{(n'',0)}_B\to
{(n',0)}_B$ such that $\operatorname{H}_0(\gamma)=\overline{\gamma}$. 
Since $\beta\gamma$ is an isomorphism by 
Lemma~\ref{lem:structure of C} $\mathrm{(1)}$, 
by replacing $\gamma$ with $\gamma{(\beta\gamma)}^{-1}$, 
we shall assume that 
$\beta\gamma=\operatorname{id}_{{(n'',0)}_B}$. 
Therefore there is a commutave diagram
$$\xymatrix{
{(n,0)}_B \ar[r]^{\alpha} \ar@{-->}[d]_{\delta} & 
{(n',0)}_B \ar[r]^{\beta} \ar@{=}[d]&
{(n'',0)}_B \ar@{=}[d]\\
{(n,0)}_B \ar@{>->}[r]^{\alpha'} & 
{(n',0)}_B \ar@{->>}[r]^{\beta} &
{(n'',0)}_B
}$$
such that the bottom line is exact. 
Here the dotted arrow $\delta$ is induced 
from the universality of $\operatorname{Ker} \beta$. 
By applying the functor $\operatorname{H}_0$ to the diagram above and 
by the five lemma, 
it turns out that $\operatorname{H}_0(\delta)$ is an isomorphism of 
projective $B/gB$-modules 
and hence $\delta$ is also an isomorphism by 
Lemma~\ref{lem:structure of C} $\mathrm{(1)}$. 
We complete the proof.
\end{proof}

\subsection{Homotopy natural transformations}
\label{subsec:Homotopy natural transformation}

\noindent
We denote the simplicial morphism 
$i^{\bigtriangleup}S_{\cdot}^{\bigtriangledown}\mathcal{C}\to
\operatorname{qis} S_{\cdot}\operatorname{Ch}_b(\mathcal{M}_B)$ 
induced from the inclusion functor 
$\eta\colon \mathcal{C}\hookrightarrow \operatorname{Ch}_b(\mathcal{M}_B)$ by 
the same letter $\eta$. 
For simplicial functors 
$\eta,\ j\mu_1,\ j\mu_2,\ 0
\colon i^{\bigtriangleup}S_{\cdot}^{\bigtriangledown}\mathcal{C}\to 
\operatorname{qis} S_{\cdot}\operatorname{Ch}_b(\mathcal{M}_B)$, 
there is a canonical natural transformation $j\mu_2\to 0$ and 
we wish to define a canonical natural transformation 
$\eta\to j\mu_1$. 
A candidate of $\eta\to j\mu_1$ is the following. 

\bigskip
\noindent
For any object ${(n,m)}_B$ in $\mathcal{C}$, we write 
$\delta_{{(n,m)}_B}\colon\eta({(n,m)}_B)\to j\mu'_1({(n,m)}_B)$ 
for the canonical projection 
$\displaystyle{\begin{bmatrix}
B^{\oplus n}\oplus B^{\oplus m}\ \ \ \ \ \ \ \ \ \ \ \\
\ \ \ \ \ \ \ \ \ \downarrow \begin{pmatrix}gE_n&0 \\ 0 & E_m \end{pmatrix}\\
B^{\oplus n}\oplus B^{\oplus m}\ \ \ \ \ \ \ \ \ \ \ 
\end{bmatrix}
\begin{matrix}
\overset{\begin{pmatrix}E_n & 0 \end{pmatrix}}{\to}
\\
\underset{\begin{pmatrix}E_n & 0 \end{pmatrix}}{\to}
\end{matrix}
\begin{bmatrix}
B^{\oplus n}\\
\downarrow gE_{n}\\
B^{\oplus n}
\end{bmatrix}}$. 
This is {\bf not} a simplicial natural transformation. 
But it has the following nice properties. 
For any morphism $\varphi=\begin{pmatrix} 
\varphi_{(n',n)} & \varphi_{(n',m)}\\
\varphi_{(m',n)} & \varphi_{(m',m)}
\end{pmatrix}\colon {(n,m)}_B\to {(n',m')}_B$ in $\mathcal{C}$, 

\begin{enumerate}
\renewcommand{\labelenumi}{$\mathrm{(\roman{enumi})}$}
\item
$\delta_{{(n,m)}_B}$ is a chain homotopy equivalence,

\item
if $\varphi$ is lower triangular, we have the equality 
$j\mu'_1(\varphi)\delta_{{(n,m)}_B}=\delta_{{(n',m')}_B}\eta(\varphi)$, 

\item
if $\varphi$ is upper triangular, 
there is a unique chain homotopy between 
$j\mu'_1(\varphi)\delta_{{(n,m)}_B}$ and $\delta_{{(n',m')}_B}\eta(\varphi)$. 
Namely since we have the equality 
$$\displaystyle{j\mu'_1(\varphi)\delta_{{(n,m)}_B}-
\delta_{{(n',m')}_B}\eta(\varphi)=
\begin{pmatrix}0 & -\varphi_{(n',m)}\\ 0 & 0 \end{pmatrix}},$$ 
the map 
$$\begin{pmatrix}0 & -\varphi_{(n',m)} \end{pmatrix}\colon 
B^{\oplus n}\oplus B^{\oplus m}\to B^{\oplus n'}$$
gives a chain homotopy between $j\mu'_1(\varphi)\delta_{{(n,m)}_B}$ and 
$\delta_{{(n',m')}_B}\eta(\varphi)$. 
$$\xymatrix{
B^{\oplus n}\oplus B^{\oplus m} 
\ar[r]^{\footnotesize{\begin{pmatrix}0 & -\varphi_{(n',m)}\end{pmatrix}}} 
\ar[d]_{\footnotesize{\begin{pmatrix}gE_n & 0 \\ 0 & E_m\end{pmatrix}}} & 
B^{\oplus n'} \ar[d]^{gE_n}\\
B^{\oplus n}\oplus B^{\oplus m} \ar[ru]_{\!\!\!\!\!\!\!\!\!\!\footnotesize{\begin{pmatrix}0 &\!\!\!\!\!\!\!\! -\varphi_{(n',m)}\end{pmatrix}}} 
\ar[r]_{\footnotesize{\begin{pmatrix}0 & -g\varphi_{(n',m)}\end{pmatrix}}} & 
B^{\oplus n'}.
}$$
\end{enumerate}
Thus we establish a theory of homotopy natural transformations. 
Then it turns out that $\delta$ induces a simplicial homotopy 
natural transformation $\eta\Rightarrow_{\operatorname{simp}}j\mu_1$. 
(For the definition of simplicial homotopy natural transformations, 
see Definition~\ref{df:simplicial hom nat trans} below.) 
Therefore by the theory below, 
it turns out that 
there exists a zig-zag sequence of 
simplicial natural transformations which connects 
$\eta$ and $j\mu_1$. 
Thus finally we obtain the fact that $\eta$ is homotopic to $0$. 
We complete the proof of zero map theorem.
\qed

\bigskip
\noindent
The rest of this subsection, we will establish a 
theory of homotopy natural transformations and 
justify the argument above.

\subsubsection*{Conventions.}
\label{conv:chain complexes conv}

\noindent
For simplicity, we set $\mathcal{E}=\operatorname{Ch}_b(\mathcal{M}_B(1))$. 
The functor $C\colon \mathcal{E} \to \mathcal{E}$ 
is given by sending a chain complex $x$ in $\mathcal{E}$ to 
$Cx:=\operatorname{Cone}\operatorname{id}_x$ the canonical mapping cone 
of the identity morphism of $x$. 
Namely the degree $n$ part of $Cx$ is ${(Cx)}_n=x_{n-1}\oplus x_n$ and 
the degree $n$ boundary morphism $d_n^{Cx}\colon {(Cx)}_n\to {(Cx)}_{n-1}$ 
is given by 
$\displaystyle{d^{Cx}_n=
\begin{pmatrix} 
-d^x_{n-1} & 0\\
-\operatorname{id}_{x_{n-1}} & d_n^x
\end{pmatrix}}$. 
For any complex $x$, 
we define 
$\iota_x\colon x \to Cx$ and 
$r_x\colon CCx \to Cx$  
to be chain morphisms 
by setting 
${(\iota_x)}_n={
\begin{pmatrix}
0\\
\operatorname{id}_{x_n}
\end{pmatrix}}$ and 
${(r_x)}_n={
\begin{pmatrix}
0 &\!\! \operatorname{id}_{x_{n-1}} &\!\! \operatorname{id}_{x_{n-1}} &\!\! 0\\
0 &\!\! 0 &\!\! 0 &\!\! \operatorname{id}_{x_n}
\end{pmatrix}}
$.

\smallskip
\noindent
Let $f$, $g\colon x\to y$ be a pair of 
chain morphisms $f$, $g\colon x\to y$ in $\mathcal{E}$. 
Recall that a {\it chain homotopy from $f$ to $g$} is 
a family of morphisms $\{h_n\colon x_n\to y_{n+1} \}_{n\in\mathbb{Z}}$ in 
$\mathcal{M}_B(1)$ indexed by the set of integers 
such that it satisfies the equality
\begin{equation}
\label{eq:chain homotopy}
d^y_{n+1}h_n+h_{n-1}d_n^x=f_n-g_n
\end{equation}
for any integer $n$. 
Then we define $H\colon Cx\to y$
to be a chain morphism by setting 
$\displaystyle{H_n:=\begin{pmatrix}-h_{n-1}& f_n-g_n \end{pmatrix}}$ 
for any integer $n$. 
We can show that the equality 
\begin{equation}
\label{eq:C-homotopy}
f-g=H\iota_x. 
\end{equation}
Conversely if we give a chain map $H\colon Cx\to y$ 
which satisfies the equality $\mathrm{(\ref{eq:C-homotopy})}$ above, 
it provides a chain homotopy from $f$ to $g$. 
We denote this situation by $H\colon f\Rightarrow_C g$ and 
we say that $H$ is a {\it $C$-homotopy from $f$ to $g$}. 
We can also show that for any complex $x$ in $\mathcal{E}$, 
$r_x$ is a $C$-homotopy from $\operatorname{id}_{Cx}$ to $0$.

\smallskip
\noindent
Let 
$[f\colon x\to x']$ and 
$[g\colon y \to y']$ be 
a pair of objects in $\mathcal{E}^{[1]}$ 
the morphisms category of $\mathcal{E}$. 
A ({\it $C$-}){\it homotopy commutative square} 
({\it from $[f\colon x\to x']$ to 
$[g\colon y\to y']$}) is a triple $(a,b,H)$ consisting of 
chain morphisms $a\colon x\to y $, $b\colon x'\to y'$ and 
$H\colon Cx\to y' $ in $\mathcal{E}$ such that 
$H\iota_x=ga-bf$. 
Namely $H$ is a $C$-homotopy from $ga$ to $bf$.

\smallskip
\noindent
Let $[f\colon x\to x']$, $[g\colon y\to y']$ and 
$[h\colon z\to z']$ be a triple of objects in $\mathcal{E}^{[1]}$ 
and let $(a,b,H)$ and $(a',b',H')$ be homotopy commutative squares 
from $[f\colon x\to x']$ to $[g\colon y\to y']$ and 
from $[g\colon y\to y']$ to $[h\colon z\to z']$ respectively. 
Then we define $(a',b',H')(a,b,H)$ to be a homotopy commutative square 
from $[f\colon x\to x']$ to $[h\colon z\to z']$ by setting
\begin{equation}
\label{eq:comp of hcd}
(a',b',H')(a,b,H):=(a'a,b'b,H'\star H)
\end{equation}
where $H'\star H$ is a $C$-homotopy from $ha'a$ to $b'bf$ 
given by the formula 
\begin{equation}
\label{eq:homotopy star}
H'\star H:=b'H+H'Ca.
\end{equation}

\smallskip
\noindent
We define $\mathcal{E}_h^{[1]}$ to be a category whose objects are 
morphisms in $\mathcal{E}$ and whose morphisms are 
homotopy commutative squares and compositions of morphisms 
are give by the formula $\mathrm{(\ref{eq:comp of hcd})}$ and 
we define $\mathcal{E}^{[1]}\to \mathcal{E}_h^{[1]}$ 
to be a functor by sending an object $[f\colon x\to x']$ to $[f\colon x\to x']$ 
and a morphism $(a,b)\colon[f\colon x\to x']\to [g\colon y\to y'] $ to 
$(a,b,0)\colon[f\colon x\to x']\to [g\colon y\to y']$.  
By this functor, we regard $\mathcal{E}^{[1]}$ as a subcategory of $\mathcal{E}_h^{[1]}$.

\smallskip
\noindent
We define $Y\colon \mathcal{E}_h^{[1]}\to\mathcal{E}$ to be a functor 
by sending an object $[f\colon x\to y]$ to $Y(f):=y\oplus Cx$ and 
a homotopy commutative square $(a,b,H)\colon [f\colon x\to y]\to 
[f'\colon x'\to y']$ to $\displaystyle{Y(a,b,H):=
\begin{pmatrix}b & -H\\ 0 &Ca \end{pmatrix}}$.

\smallskip
\noindent
We write $s$ and $t$ for the functors $\mathcal{E}_h^{[1]}\to \mathcal{E} $  
which sending an object $[f\colon x\to y] $ to $x$ and $y$ respectively. 
We define $j_1\colon s\to Y$ and $j_2\colon t \to Y$ to be natural 
transformations by setting 
${j_1}_f:=\begin{pmatrix}f\\ -\iota_x \end{pmatrix}  $ and 
${j_2}_f:=\begin{pmatrix}\operatorname{id}_y\\ 0\end{pmatrix}$ respectively for any object $[f\colon x\to y]$ in $\mathcal{E}_h^{[1]}$.

\begin{df}[Homotopy natural transformations]
\label{df:homotopy natural transformation}
Let $\mathcal{I}$ be a category and let $f$, $g\colon \mathcal{I}\to\mathcal{E}$ 
be a pair of functors. 
A {\it homotopy natural transformation} ({\it from $f$ to $g$}) 
is consisting of a family of morphims 
$\{\theta_i\colon f_i\to g_i\}_{i\in\operatorname{Ob}\mathcal{I}}$ 
indexed by the class of objects of $\mathcal{I}$ and 
a family of $C$-homotopies 
$\{\theta_a\colon g_a\theta_a\Rightarrow_C 
\theta_jf_a\}_{a\colon i\to j\in \operatorname{Mor} \mathcal{I}}$ 
indexed by the class of morphisms of $\mathcal{I}$ 
such that for any object $i$ of $\mathcal{I}$, 
$\theta_{{\operatorname{id}_i}}=0$ and 
for any pair of composable morphisms $i\overset{a}{\to}j\overset{b}{\to}k$ 
in $\mathcal{I}$, 
$\theta_{ba}=\theta_b\star\theta_a(=g_b\theta_a+\theta_bCf_a)$. 
We denote this situation by $\theta\colon f\Rightarrow g$. 
For a usual natural transformation $\kappa\colon f\to g$, 
we regard it as a homotopy natural transformation 
by setting $\kappa_{a}=0$ for any morphism $a\colon i\to j$ in $\mathcal{I}$.

\smallskip
\noindent
Let $h$ and $k$ be another functors from $\mathcal{I}$ to $\mathcal{E}$ and 
let $\alpha\colon f\to g$ and $\gamma\colon h\to k$ be 
natural transformations and $\beta\colon g\Rightarrow h$ a homotopy natural 
transformation. 
We define $\beta\alpha\colon f\Rightarrow h$ and 
$\gamma\beta\colon g\Rightarrow k$ to be homotopy natural transformations 
by setting for any object $i$ in $\mathcal{I}$, 
${(\beta\alpha)}_i=\beta_i\alpha_i$ and 
${(\gamma\beta)}_i=\gamma_i\beta_i$ and 
for any morphism $a\colon i\to j$ in $\mathcal{I}$, 
${(\beta\alpha)}_a:=\beta_aC(\alpha_i)$ and 
${(\gamma\beta)}_a=\gamma_j\beta_a$. 
\end{df}

\begin{ex}[Homotopy natural transformations]
\label{ex:homotopy natural transformation}
We define $\epsilon\colon s\Rightarrow t$ and 
$p\colon Y\Rightarrow t$ to be homotopy natural transformations 
between functors $\mathcal{E}_h^{[1]}\to \mathcal{E}$ 
by setting for any object $[f\colon x\to y]$ in $\mathcal{E}_h^{[1]}$, 
$\epsilon_{f}:=f\colon x\to y$ and 
$p_{f}:=\begin{pmatrix}\operatorname{id}_y & 0\end{pmatrix}\colon 
Y(f)=y\oplus Cx\to y$ 
and for a homotopy commutative square 
$(a,b,H)\colon [f\colon x\to y]\to [f'\colon x'\to y']$, 
$\epsilon_{(a,b,H)}:=H\colon f'a\Rightarrow_C bf$ and 
$p_{(a,b,H)}:=\begin{pmatrix}0 & -Hr_x\end{pmatrix}\colon 
b\begin{pmatrix}\operatorname{id}_y & 0\end{pmatrix}\Rightarrow_C 
\begin{pmatrix}\operatorname{id}_{y'} & 0\end{pmatrix}
\begin{pmatrix}b & -H\\ 0 & Ca \end{pmatrix}$. 
Then we have the commutative diagram of homotopy natural transformations.
$$\xymatrix{
s \ar[r]^{{j_1}} \ar[rd]_{\epsilon} & Y \ar[d]_{p} & t \ar[l]_{{j_2}} \ar[ld]^{\operatorname{id}_t}\\
& t. 
}$$
Here we can show that for any object $[f\colon x\to  y]$ in $\mathcal{E}_h^{[1]}$, $p_f$ and ${j_2}_f$ are chain homotopy equivalences. 
In particular if $f$ is a chain homotopy equivalence, then 
${j_1}_f$ is also a chain homotopy equivalence.
\end{ex}

\begin{df}[Mapping cylinder functor on $\operatorname{Nat}_h(\mathcal{E}^{\mathcal{I}})$]
\label{df:Nat_h}
Let $\mathcal{I}$ be a small category. 
We will define $\operatorname{Nat}_h(\mathcal{E}^{\mathcal{I}})$ 
{\it the category of homotopy natural transformations} 
({\it between the functors from $\mathcal{I}$ to $\mathcal{E}$}) as follows. 
An object in $\operatorname{Nat}_h(\mathcal{E}^{\mathcal{I}})$ 
is a triple $(f,g,\theta)$ consisting of 
functors $f$, $g\colon \mathcal{I}\to \mathcal{E}$ and 
a homotopy natural transformation $\theta\colon f\Rightarrow g$. 
A morphism $(a,b)\colon (f,g,\theta)\to (f',g',\theta')$ 
is a pair of natural transformations $a\colon f\to f'$ and $b\colon g\to g'$ 
such that $\theta'a=b\theta$. 
Compositions of morphisms is given by componentwise compositions of natural 
transformations.

\smallskip
\noindent
We will define functors $S,\ T,\ Y\colon\operatorname{Nat}_h
(\mathcal{E}^{\mathcal{I}})\to \mathcal{E}^{\mathcal{I}}$ and 
natural transformations $J_1\colon S\to Y$ and $J_2\colon T\to Y$ as follows. 
For an object $(f,g,\theta)$ and a morphism 
$(\alpha,\beta)\colon (f,g,\theta)\to (f',g',\theta')$ 
in $\operatorname{Nat}_h(\mathcal{E}^{\mathcal{I}})$ and 
an object $i$ and 
a morphism $a\colon i\to j$ in $\mathcal{I}$, we set 
$$S(f,g,\theta):=f,\ S(\alpha,\beta):=\alpha,$$ 
$$T(f,g,\theta):=g,\ T(\alpha,\beta):=\beta,$$ 
$${Y(f,g,\theta)}_i(={Y(\theta)}_i):=Y(\theta_i)(=g_i\oplus C(f_i)),$$ 
$$\displaystyle{{Y(f,g,\theta)}_a(={Y(\theta)}_a):=Y(f_a,g_a,\theta_a)
\left(=\begin{pmatrix}g_a & -\theta_a\\ 0& Cf_a \end{pmatrix}\right)},\ 
\displaystyle{{Y(\alpha,\beta)}_i:=\begin{pmatrix}\beta_i & 0\\ 0 & C(\alpha_i) \end{pmatrix}},$$
$${{J_1}_{(f,g\theta)}}_i:={j_1}_{\theta_i},\ {{J_2}_{(f,g\theta)}}_i:={j_2}_{\theta_i}.$$
In particular 
for an object $(f,g,\theta)$ in $\operatorname{Nat}_h
(\mathcal{E}^{\mathcal{I}})$ 
if for any object $i$ of $\mathcal{I}$, 
$\theta_i$ is a chain homotpy equivalence, 
then there exists a zig-zag diagram which connects $f$ to $g$, 
$f\overset{J_1}{\to} Y(\theta) \overset{J_2}{\leftarrow} g$ such that 
for any object $i$, ${J_1}_i$ and ${J_2}_i$ are chain homotopy equivalences.
\end{df}

\begin{df}[Simplicial homotopy natural transformations]
\label{df:simplicial hom nat trans}
Let $\mathcal{J}$ 
be a simplicial object in the category of small categories which we 
call shortly a simplicial small category 
and 
let $f$, $g\colon \mathcal{J}\to \operatorname{Ch}_b(S_{\cdot}\mathcal{M}_B(1))$ 
be simplicial functors. 
Recall that a {\it simplicial natural transformaion} 
({\it from $f$ to $g$}) is 
a family of natural transformations $\{\rho_n\colon f_n\to g_n\}_{n\geq 0}$ 
indexed by non-negative integers such that for any morphism 
$\varphi\colon [n]\to [m]$, 
we have the equality $\rho_nf_{\varphi}=g_{\varphi}\rho_m$. 

\smallskip
\noindent
A {\it simplicial homotopy natural transformation} 
({\it from $f$ to $g$}) is a family of 
homotopy natural transformations 
$\{\theta_n\colon f_n\Rightarrow g_n\}_{n\geq 0}$ 
indexed by non-negative integers such that 
for any morphism $\varphi\colon [n]\to [m]$, 
we have the equality $\theta_nf_{\varphi}=g_{\varphi}\theta_m$. 
We denote this situation by $\theta\colon f\Rightarrow_{\operatorname{simp}} g$. 

\smallskip
\noindent
For a simplicial homotopy natural transformation 
$\theta\colon f\Rightarrow_{\operatorname{simp}} g$, we will define 
$\mathcal{Y}(\theta)\colon \mathcal{J}\to \operatorname{Ch}_b(S_{\cdot}\mathcal{M}_B(1))$ and 
$\mathfrak{J}_1\colon f\to \mathcal{Y}(\theta)$ and 
$\mathfrak{J}_2\colon g\to \mathcal{Y}(\theta)$ 
to be a simplicial functor 
and simplicial natural transformations respectively as follows. 
For any $[n]$ and any morphism $\varphi\colon [m]\to [n]$, 
we set 
${\mathcal{Y}(\theta)}_n:=Y(\theta_n)$, ${\mathfrak{J}_1}_n:={J_1}_{\theta_n}$, 
${\mathfrak{J}_2}_n:={J_2}_{\theta_n}$  and 
${\mathcal{Y}(\theta)}_{\varphi}:=Y(f_{\varphi},g_{\varphi})$. 
In particular 
if for any non-negative integer $n$, 
any object $j$ of $\mathcal{J}_n$, ${\theta_n}_j$ 
is a chain homotpy equivalence, 
then there exists a zig-zag sequence which connects $f$ to $g$, 
$f\overset{\mathfrak{J}_1}{\to} \mathcal{Y}(\theta) 
\overset{\mathfrak{J}_2}{\leftarrow} g$ such that 
for any non-negative integer $n$ and any object $j$, ${{J_1}_n}_j$ and 
${{J_2}_n}_j$ are chain homotopy equivalences.
\qed
\end{df}

\section{Motivic Gersten's conjecture}

In this last section, I briefly explain 
my recent work of motivic Gernsten's conjecture 
in \cite{Moc16b}.

\subsection{Motivic zero map theorem}
\label{subsec:Motivic zero map theorem}

We denote the category of small dg-categories over $\mathbb{Z}$ 
the ring of integers 
by $\operatorname{dgCat}$ and let 
${\mathcal{M}\!\!ot}^{\operatorname{add}}_{\operatorname{dg}}$ and 
${\mathcal{M}\!\!ot}^{\operatorname{loc}}_{\operatorname{dg}}$ 
be symmetric monoidal stable 
model categories of additve noncommutative motives and 
localizing motives over $\mathbb{Z}$ respectively. 
(See \cite[\S 7]{CT12}.) 
We denote the homotopy category of a model category $\mathcal{M}$ by 
$\operatorname{Ho}(\mathcal{M})$. 
There are functors from $\operatorname{ExCat}$ the 
category of small exact categories to $\operatorname{dgCat}$ 
which sending a small exact category $\mathcal{E}$ to 
its bounded dg-derived category $\mathcal{D}_{\operatorname{dg}}^b(\mathcal{E})$ 
(see \cite[\S 4.4]{Kel06}.) 
and the universal functors 
$\mathcal{U}_{\operatorname{add}}\colon\operatorname{dgCat}\to 
\operatorname{Ho}({\mathcal{M}\!\!ot}^{\operatorname{add}}_{\operatorname{dg}})$ 
and 
$\mathcal{U}_{\operatorname{loc}}\colon\operatorname{dgCat}\to 
\operatorname{Ho}
({\mathcal{M}\!\!ot}^{\operatorname{loc}}_{\operatorname{dg}})$. 
We denote the compositions of these functors 
$\operatorname{ExCat}\to
\operatorname{Ho}({\mathcal{M}\!\!ot}^{\operatorname{add}}_{\operatorname{dg}})$ and 
$\operatorname{ExCat}\to\operatorname{Ho}({\mathcal{M}\!\!ot}^{\operatorname{loc}}_{\operatorname{dg}})$ 
by $M_{\operatorname{add}}$ and $M_{\operatorname{loc}}$ respectively. 
Then we can ask the following question:

\begin{con}[Motivic Gersten's conjecture provisonal version]
\label{conj:motivic gersten conj provisonal}
For any commutative regular local ring $R$ and for any integers 
$1\leq p\leq \dim R$, the inclusion functor 
$\mathcal{M}_R^{p-1}\hookrightarrow \mathcal{M}_R^p$ induces the zero morphism 
$$M_{\#}(\mathcal{M}_R^{p-1})\to M_{\#}(\mathcal{M}_R^p)$$
in $\operatorname{Ho}({\mathcal{M}\!\!ot}^{\#}_{\operatorname{dg}})$ 
where $\#\in\{\operatorname{add}, \operatorname{loc}\}$.
\end{con}

\smallskip
\noindent
Since in the proof of zero morphism theorem in \S~\ref{sec:zero map theorem}, 
what we just using are basically the resolution theorem and 
the additivity theorem, 
by mimicking the proof of zero morphism theorem, 
we obtain the following result:

\begin{prop}[Motivic zero morphism theorem]
\label{prop:motivic zero map theorem}
Let $B$ be a commutative noetherian ring with $1$ and let $g$ 
be an element in $B$ such that $g$ is a non-zero divisor and 
contained in the Jacobson radical of $B$. 
Assume that every finitely generated projective $B$-modules are free. 
Then the inclusion functor $\mathcal{P}_{B/gB}\hookrightarrow 
\mathcal{M}_B(1)$ induces the zero morphism 
$$M_{\#}(\mathcal{P}_{B/gB})\to M_{\#}(\mathcal{M}_B(1))$$
in ${\mathcal{M}\!\!ot}^{\#}_{\operatorname{dg}}$ 
where $\#\in \{\operatorname{add}, \operatorname{loc}\}$.
\qed
\end{prop}

\begin{rem}
\label{rem:motivic zero map theorem}
Strictly saying, in the proof above, 
we shall moreover utilize the following two facts.

\begin{enumerate}
\item
For any pair of exact functors $f$, $g\colon \mathcal{E}\to
\operatorname{Ch}_b(\mathcal{F})$ 
from an exact category $\mathcal{E}$ to 
the category of bounded chain complexes on an exact category $\mathcal{F}$, 
if there exists a zig-zag sequence 
of natural equivalences which connects $f$ and $g$ and 
whose components are quasi-isomorphisms, 
then $f$ and $g$ induce the same morphisms $M_{\#}(f)=M_{\#}(g)$ 
in ${\mathcal{M}\!\!ot}^{\#}_{\operatorname{dg}}$ 
where $\#\in\{\operatorname{add},\ \operatorname{loc}\}$.

\item
Let $\Delta$ be the simplicial category and let $e$ be the final object 
in the category of small categories and let 
$p\colon \Delta \to e$ be the projection functor. 
We denote the derivator associated with the model category 
${\mathcal{M}\!\!ot}^{\operatorname{loc}}_{\operatorname{dg}}$ by 
$\operatorname{Mot}^{\operatorname{loc}}_{\operatorname{dg}}$. 
For a small dg category $\mathcal{A}$, there is an 
object $\mathcal{U}_{\operatorname{loc}}(S_{\cdot}\mathcal{A})$ in 
$\operatorname{Mot}^{\operatorname{loc}}_{\operatorname{dg}}(\Delta)$ 
where $S_{\cdot}$ is the 
Segal-Waldhausen $S$-construction. 
We have the canonical isomorphism in 
$\operatorname{Ho}({\mathcal{M}\!\!ot}^{\operatorname{loc}}_{\operatorname{dg}})=
\operatorname{Mot}^{\operatorname{loc}}_{\operatorname{dg}}(e)$
$$p_{!}\mathcal{U}_{\operatorname{loc}}(S_{\cdot}\mathcal{A})=
\mathcal{U}_{\operatorname{loc}}(\mathcal{A})[1].$$
(See \cite[Notation 11.11, Theorem 11.12]{Tab08}.)

\end{enumerate}

\end{rem}

\subsection{Nilpotent invariant motives}
\label{subsec:Nilpotent invariant motives}

A problem to obtain the motivic Gersten's conjecture from 
Proposition~\ref{prop:motivic zero map theorem} is that in the proof 
in \S~\ref{sec:Strategy of proof}, 
we utilze the d\'evissage theorem. (see the isomorphism 
$\mathrm{(\ref{eq:Gersten strategy 2})}$ \textbf{III}.) 
Thus if we imitate the proof, then we need to modify 
statement of the conjecture 
and we need to construct a new stable model category of 
non-commutative motives 
which we denote by 
${\mathcal{M}\!\!ot}^{\operatorname{nilp}}_{\operatorname{dg}}$ 
and which we construct by localizing 
${\mathcal{M}\!\!ot}^{\operatorname{loc}}_{\operatorname{dg}}$. 

\smallskip
\noindent
We briefly recall the construction of the stable model category 
of localizing non-commutative motives 
${\mathcal{M}\!\!ot}^{\operatorname{loc}}_{\operatorname{dg}}$ 
from \cite{Tab08}. 
First notice that 
the category $\operatorname{dgCat}$ carries a cofibrantly generated 
model structure whose weak equivalences are the derived Morita equivalences. 
\cite[Th\'eor\`eme 5.3]{Tab05}. 
We fix on $\operatorname{dgCat}$ a fibratnt resolution functor $R$, 
a cofibrant resolution functor $Q$ and a left framing 
$\Gamma_{\ast}$ (see \cite[Definition 5.2.7 and Theorem 5.2.8]{Hov99}.) 
and we also fix a small full subcategory 
$\operatorname{dgCat}_f\hookrightarrow\operatorname{dgCat}$ 
such that it contains all finite dg cells and 
any objects in $\operatorname{dgCat}_f$ are ($\mathbb{Z}$-)flat 
and homotopically finitely presented (see \cite[Definition 2.1 (3)]{TV07}.) 
and $\operatorname{dgCat}_f$ is closed uder the operations $Q$, $QR$ and 
$\Gamma_{\ast}$ and $\otimes$. 
The construction below does not depend upon a choice of $\operatorname{dgCat}_f$ 
up to Dwyer-Kan equivalences. 

\smallskip
\noindent
Let $s\widehat{\operatorname{dgCat}_f}$ and 
$s\widehat{{\operatorname{dgCat}_f}_{\bullet}}$ be the category of simplicial presheaves and that of pointed simplicial 
preshaves on $\operatorname{dgCat}_f$ respectively. 
We have the projective model structures on $s\widehat{\operatorname{dgCat}_f}$ 
and $s\widehat{{\operatorname{dgCat}_f}_{\bullet}}$ 
where the weak equivalences and the fibrations are 
the termwise simplicial weak equivalences and termwise Kan fibrations 
respectively. (see \cite[Theorem 11.6.1]{Hir03}.) 

\smallskip
\noindent
We denote the class of derived Morita equivalences in 
$\operatorname{dgCat}_f$ by $\Sigma$ and 
we also write $\Sigma_{+}$ for the image of $\Sigma$ 
by the composition of the Yoneda embedding 
$h\colon \operatorname{dgCat}_f \to s\widehat{\operatorname{dgCat}}$ and 
the canonical functor 
${(-)}_{+}\colon s\widehat{\operatorname{dgCat}_f}\to 
s\widehat{{\operatorname{dgCat}_f}_{\bullet}}$. 

\smallskip
\noindent
Let $P$ be the canonical map $\emptyset \to h(\emptyset)$ in 
$s\widehat{\operatorname{dgCat}_f}$ and we write $P_{+}$ 
for the image of $P$ by the functor ${(-)}_{+}$. 
We write $L_{\Sigma,P}s\widehat{{\operatorname{dgCat}_f}_{\bullet}}$ 
for the left Bosufield localization of 
$s\widehat{{\operatorname{dgCat}_f}_{\bullet}}$ by the 
set $\Sigma_{+}\cup\{P_{+}\}$. 
The Yoneda embedding functor induces a functor 
$$\mathbb{R}\underline{h}\colon \operatorname{Ho}(\operatorname{dgCat})\to\operatorname{Ho}(L_{\Sigma,P}s\widehat{{\operatorname{dgCat}_f}_{\bullet}})$$
which associates any dg category $\mathcal{A}$ to 
the pointed simplicial presheaves on $\operatorname{dgCat}_f$:
$$\mathbb{R}\underline{h}(\mathcal{A})\colon\mathcal{B}\mapsto 
{\operatorname{Hom}(\Gamma_{\ast}(Q\mathcal{B}),R(\mathcal{A}))}_{+}.$$

\smallskip
\noindent
Let $\mathcal{E}$ be the class of morphisms in 
$L_{\Sigma,P}s\widehat{{\operatorname{dgCat}_f}_{\bullet}}$ of shape 
$$\operatorname{Cone}[\mathbb{R}\underline{h}(\mathcal{A})\to 
\mathbb{R}\underline{h}(\mathcal{B})]\to \mathbb{R}\underline{h}(\mathcal{C})$$
associated to each exact sequence of dg categories 
$$\mathcal{A}\to\mathcal{B}\to\mathcal{C},$$
with $\mathcal{B}$ in $\operatorname{dgCat}_f$ 
where $\operatorname{Cone}$ means homotopy cofiber. 
We write ${\mathcal{M}\!\!ot}_{\operatorname{dg}}^{\operatorname{uloc}}$ 
for the left Bosufield localization of $L_{\Sigma,P}s\widehat{{\operatorname{dgCat}_f}_{\bullet}}$ by $\mathcal{E}$ 
and call it the {\it model category of 
unstable localizing non-commutative motives}. 

\smallskip
\noindent
Finally we write ${\mathcal{M}\!\!ot}_{\operatorname{dg}}^{\operatorname{loc}}$ 
for the stable symmetric monoidal model 
category of symmetric $S^1\otimes 1\!\!\!1$-spectra on ${\mathcal{M}\!\!ot}_{\operatorname{dg}}^{\operatorname{uloc}} $ 
(see \cite[\S 7]{Hov01}.) and call it 
the {\it model category of localizing non-commutative motives}. 

\bigskip
\noindent
Next we construct the stable model category ${\mathcal{M}\!\!ot}_{\operatorname{dg}}^{\operatorname{nilp}} $. 
First recall that 
we say that a non-empty full subcategory $\mathcal{Y}$ of 
a Quillen exact category $\mathcal{X}$ 
is a {\it topologizing subcategory} of $\mathcal{X}$ if 
$\mathcal{Y}$ is closed under finite direct sums and closed under 
admissible sub- and quotient objects. 
The naming of the term `topologizing' comes from 
noncommutative geometry of abelian categories 
by Rosenberg. (See \cite[Lecture 2 1.1]{Ros08}.) 
We say that a full subcategory $\mathcal{Y}$ of 
an exact category $\mathcal{X}$ is a {\it Serre subcategory} 
if it is an extensional closed topologizing subcategory of $\mathcal{X}$. 
For any full subcategory $\mathcal{Z}$ of $\mathcal{X}$, 
we write ${}^S\!\!\!\sqrt{\mathcal{Z}}$ 
for intersection of all Serre subcategories 
which contain $\mathcal{Z}$ and call it the 
{\it Serre radical of $\mathcal{Z}$} 
({\it in $\mathcal{X}$}). 

\begin{df}[Nilpotent immersion]
\label{df:nilp immersion}
Let $\mathcal{A}$ 
be a noetherian abelian category and let 
$\mathcal{B}$ a topologizing subcategory. 
We say that {\it $\mathcal{B}$ satisfies the d\'evissage condition} 
({\it in $\mathcal{A}$}) or say that the inclusion 
$\mathcal{B}\hookrightarrow\mathcal{A}$ 
is a {\it nilpotent immersion}
if one of the following equivalent conditions holds:
\begin{enumerate}
\item
For any object $x$ in $\mathcal{A}$, 
there exists a finite filtration of 
monomorphisms 
$$x=x_0\leftarrowtail x_1\leftarrowtail x_2\leftarrowtail \cdots 
\leftarrowtail x_n=0$$
such that for every $i<n$, $x_i/x_{i+1}$ is 
isomorphic to an object in $\mathcal{B}$. 

\item 
We have the equality 
$$\mathcal{A}={}^S\!\!\!\sqrt{\mathcal{B}}.$$
\end{enumerate}
(For the proof of the equivalence of the conditions above, 
see \cite[3.1]{Her97}, \cite[2.2]{Gar09}.) 
\end{df}

\begin{df}
\label{df:nilp motives}
We write $\mathcal{N}$ for the class of morphisms in 
${\mathcal{M}\!\!ot}_{\operatorname{dg}}^{\operatorname{uloc}}$ of shape 
$$\mathbb{R}\underline{h}(\mathcal{D}_{\operatorname{dg}}^b(\mathcal{B}))\to\mathbb{R}\underline{h}(\mathcal{D}_{\operatorname{dg}}^b(\mathcal{A}))$$
asscoated with each noetherian abelian category $\mathcal{A}$ 
and each nilpotent immersion $\mathcal{B}\hookrightarrow\mathcal{A}$. 

\smallskip
\noindent
We write ${\mathcal{M}\!\!ot}_{\operatorname{dg}}^{\operatorname{unilp}}$ 
the left Bousfield localization of 
${\mathcal{M}\!\!ot}_{\operatorname{dg}}^{\operatorname{uloc}}$ by 
$\mathcal{N}$ and 
call it the 
{\it model category of unstable nilpotent invariant non-commutative motives}. 

\smallskip
\noindent
Finally we write ${\mathcal{M}\!\!ot}_{\operatorname{dg}}^{\operatorname{nilp}}$ 
for the stable model 
category of symmetric $S^1\otimes 1\!\!\!1$-spectra on 
${\mathcal{M}\!\!ot}_{\operatorname{dg}}^{\operatorname{unilp}}$ 
and call it 
the {\it stable model category of nilpotent invariant non-commutative motives}. 
We denote the compositons of the following functors
$$
\operatorname{ExCat}\overset{\mathcal{D}^b_{\operatorname{dg}}}{\to}
\operatorname{dgCat}\to\operatorname{Ho}(\operatorname{dgCat})
\overset{\mathbb{R}\underline{h}}{\to} 
\operatorname{Ho}(L_{\Sigma,P}s\widehat{{\operatorname{dgCat}_f}_{\bullet}})\to
\operatorname{Ho}({\mathcal{M}\!\!ot}_{\operatorname{dg}}^{\operatorname{unilp}})\overset{\Sigma^{\infty}}{\to}
\operatorname{Ho}({\mathcal{M}\!\!ot}_{\operatorname{dg}}^{\operatorname{nilp}})
$$
by $M_{\operatorname{nilp}}$.
\end{df}

\smallskip
\noindent
We obtain the following theorem.

\begin{thm}[Motivic Gersten's conjecture]
\label{thm:motivic Gersten conjecture}
For any commutative regular local ring $R$ and for any integers 
$1\leq p\leq \dim R$, the inclusion functor 
$\mathcal{M}_R^{p-1}\hookrightarrow \mathcal{M}_R^p$ induces the zero morphism 
$$M_{\operatorname{nilp}}(\mathcal{M}_R^{p-1})\to M_{\operatorname{nilp}}(\mathcal{M}_R^p)$$
in $\operatorname{Ho}({\mathcal{M}\!\!ot}^{\operatorname{nilp}}_{\operatorname{dg}})$. 
\qed
\end{thm}

\medskip
\noindent
SATOSHI MOCHIZUKI\\
{\it{DEPARTMENT OF MATHEMATICS,
CHUO UNIVERSITY,
BUNKYO-KU, TOKYO, JAPAN.}}\\
e-mail: {\tt{mochi@gug.math.chuo-u.ac.jp}}\\


\begin{thebibliography}{1234567}
\bibitem[Bal09]{Bal09}
P.~Balmer, 
\emph{Niveau spectral sequences on singular schemes and failure of
  generalized {G}ersten conjecture}, 
  Proc.\ Amer.\ Math.\ Soc.\ \textbf{137} (2009), 
p.99-106.

\bibitem[BW02]{BW02}
P.~Balmer and C.~Walter, 
\emph{A Gersten-Witt spectral sequence for regular schemes}, 
Annales Scientifiques de l'\'Ecole Normale Sup\'erieure 
Vol.\textbf{35} (2002), 
p.127-152.

\bibitem[BGPW03]{BGPW03}
P.~Balmer, S.~Gillet, I.~Panin and C.~Walter, 
\emph{The Gersten conjecture for Witt group in the equicharacteristic case}, 
Documenta Math. \textbf{8} (2003), 
p.203-217.


\bibitem[Bei84]{Bei84}
A.~ A.~Beilinson, 
\emph{Higher regulators and values of $L$-functions}, 
Current problems in mathematics, \textbf{24} 
Itogi Nauki i Tekhniki, Akad. Nauk SSSR, 
Vsesoyuz. Inst. Nauchn. i Tekhn. Inform., Moscow (1984), 
p.181-238.


\bibitem[Bei87]{Bei87}
A.~A.~Beilinson, 
\emph{Height pairing between algebraic cycles}, 
In: $K$-theory, Arithmetic and Geometry, Moscow, 1984-1986, 
Lect. Notes in Math., \textbf{1289}, 
Springer (1987), 
p.1-25.



\bibitem[Blo74]{Blo74}
S.~Bloch, 
\emph{$K_2$ and algebraic cycles}, 
Ann. of Math. \textbf{99} (1974), 
p.266-292. 

\bibitem[Blo86]{Blo86}
S.~Bloch, 
\emph{{A} note on {G}ersten's conjecture in the mixed characteristic case}, 
Contemporary Math, \textbf{55}, Part I (1986), 
p.75-78. 

\bibitem[BO74]{BO74}
S.~Bloch and A.~Ogus, 
\emph{Gersten's conjecture and the homology of schemes}, 
Ann. Sci. Ecole Norm. Sup. \textbf{7} (1974), 
p.181-202.

\bibitem[Bou64]{Bou64}
N.~Bourbaki, 
{\it{Alg\`ebre commutative Chapitres 5 \`a 7}}, 
Hermann/Addison-Wesley (1964).

\bibitem[BW16]{BW16}
O.~Braunling and J.~Wolfson, 
\emph{Hochschild coniveau spectral sequence and the Beilinson residue}, 
arXiv:1607.07756v1 (2016).

\bibitem[BG73]{BG73}
K.~Brown and S.~Gersten, 
\emph{Algebraic $K$-theory as generalized sheaf cohomology}, 
Higher $K$-theories, Springer Lect. Notes Math \textbf{341} (1973), 
p.266-292.

\bibitem[CT12]{CT12}
D-C Cisinski and G. Tabuada, 
\emph{Symmetric monoidal structure on non-commutative motives}, 
Journal of $K$-theory, \textbf{9} (2012), 
p.201-268.

\bibitem[CF68]{CF68}
L.~Claborn and R.~Fossum, 
\emph{{G}eneralizations of the notion of class group}, 
Ill. Jour. Math. \textbf{12} (1968), 
p.228-253.

\bibitem[C-T95]{C-T95}
J.-L.~Colliot-Th\'el\`ene, 
\emph{Birational invarinats, purity and the Gersten conjecture}, 
in $K$-theory and algebraic geometry: connection with quadratic 
forms and division algebras, (W. Jacob and A. Rosenberg, ed.), 
Proceedings of Symposia in Pure Mathematics \textbf{58} (1995), 
p.1-64.

\bibitem[C-THK97]{C-THK97}
J.-L.~Colliot-Th\'el\`ene, R.~Hoobler and B.~Kahn, 
\emph{The Bloch-Ogus-Gabber theorem}, 
(Tronto 1996), Fields Inst. Commun. \textbf{16} (1997), 
p.31-94.

\bibitem[Dah15]{Dah15}
C.~Dahlhausen, 
\emph{Milnor $K$-theory of complete DVR's with finite residue field}, 
arXiv:1509.0187v1 (2015)

\bibitem[DS72]{DS72}
K.~Dennis and M.~Stein, 
\emph{A new exact sequence for $K_2$ and some consequences for rings of integers}, 
Bull. Amer.Math. Soc. \textbf{78} (1972), 
p.600-603. 


\bibitem[DS75]{DS75}
K.~Dennis and M.~Stein, 
\emph{$K_2$ of discrete valuation rings}, 
Adv. in Math. \textbf{18} (1975), 
p.182-238.

\bibitem[Die86]{Die86}
V.~Diekert, 
\emph{{E}ine {B}emerkung zu freien moduln \"uber regul\"aren lokalen ringen}, 
Journal of Algebra \textbf{101} (1986), 
p.188-189.


\bibitem[DHY15]{DHY15}
B.~F.~Dribus, J.~W.~Hoffman, S.~Yang, 
\emph{Infinitesimal theory of Chow groups via $K$-theory, Cyclic 
homology and the relative Chern character}, 
arXiv:1501.07525 (2015).

\bibitem[DHM85]{DHM85}
S.~P.~Dutta, M.~Hochster and J.~E.~McLaughlin, 
\emph{{M}odules of finite projective dimension with 
negative intersection multiplicities}, 
Invent. Math. \textbf{79} (1985), p.253-291.

\bibitem[Dut93]{Dut93}
S.~P.~Dutta, 
\emph{{A} note on {C}how groups and intersection multiplicity of {M}odules}, 
Journal of Algebra \textbf{161} (1993), 
p.186-198.

\bibitem[Dut95]{Dut95}
S.~P.~Dutta, 
\emph{{O}n {C}how groups and intersection multiplicity of {M}odules, II}, 
Journal of Algebra \textbf{171} (1995), 
p.370-381.

\bibitem[FS02]{FS02}
E.~Friedlander and A.~Suslin, 
\emph{The spectral sequence relating algebraic $K$-theory 
to motivic cohomology}, 
Annales Scientifiques de l`\'Ecole Normale Sup\'erieure 
Vol. \textbf{35} (2002), 
p.773-875.

\bibitem[Gab93]{Gab93}
O.~Gabber, 
\emph{An injectivity property for \'etale cohomology}, 
Compositio Math. \textbf{86}, 
p.1-14.

\bibitem[Gab94]{Gab94}
O.~Gabber, 
\emph{Gersten's conjecture for some complexes of vanishing cycles}, 
manuscripta mathematica Vol.\textbf{85} (1994), 
p.323-343.

\bibitem[Gar09]{Gar09}
G.~A.~Garkusha, 
\emph{{C}lassification of finite localizations of quasi-coherent sheaves. (Russian)}, 
Algebra i Analiz \textbf{21} (2009), 
p.93-129; 
translation in St. Petersburg Math. J. \textbf{21} (2010), 
p.433-458.

\bibitem[Gei98]{Gei98}
T.~Geisser, 
\emph{Tate's conjecture, algebraic cycles and rational $K$-theroy 
in characteristic $p$}, 
$K$-theory \textbf{13} (1998), 
p.109-122.

\bibitem[Gei04]{Gei04}
T.~Geisser, 
\emph{Motivic cohomology over Dedekind rings}, 
Math. Z., \textbf{248} (2004), 
p.773-794.

\bibitem[GL00]{GL00}
T.~Geisser and M.~Levine, 
\emph{{T}he {$K$}-theory of fields in characteristic $p$}, 
Invent. Math. \textbf{139} (2000), 
p.459-493.

\bibitem[Ger73]{Ger73} 
S.~Gersten, 
\emph{{S}ome exact sequences in the higher {$K$}-theory of rings}, 
In Higher {$K$}-theories, 
Springer Lect. Notes Math. \textbf{341} (1973), 
p.211-243.


\bibitem[Gil86]{Gil86}
H.~Gillet, 
\emph{{G}ersten's conjecture for the {$K$}-theory with torsion coefficients of a discrete valuation ring}, 
Jour. Algebra \textbf{103} (1986), 
p.377-380.


\bibitem[GL87]{GL87}
H.~Gillet and M.~Levine, 
\emph{{T}he relative form of {G}ersten's conjecture over a 
discrete valuation ring:The smooth case}, 
Jour. Pure. App. Alg. \textbf{46} (1987), 
p.59-71.

\bibitem[GS87]{GS87}
H.~Gillet and C.~Soul\'e, 
\emph{{I}ntersection theory using {A}dams operations}, 
Invent. Math. \textbf{90} (1987), 
p.243-277.

\bibitem[GS99]{GS99}
H.~Gillet and C.~Soul\'e, 
\emph{{F}iltrations on higher algebraic $K$-theory}, 
Algebraic $K$-theory (Seatle, WA. 1997) 
Proc. Sympos. Pure Math. \textbf{67} (1999), 
p.89-148.

\bibitem[Gra85]{Gra85}
D.~R.~Grayson, 
\emph{Universal exactness in algebraic $K$-theory}, 
J. Pure Appl. Algebra \textbf{36} (1985), 
p.139-141.


\bibitem[Gra92]{Gra92}
D.~R.~Grayson, 
\emph{{A}dams operations on higher {$K$}-theory}, 
$K$-theory \textbf{6} (1992), 
p.97-111.

\bibitem[GG05]{GG05}
M.~Green and P.~Griffiths, 
\emph{On the tangent space to the space of algebraic cycles 
on a smooth algebraic variety}, 
Annals of Mathematics Studies \textbf{157}, 
Princeton University Press (2005).

\bibitem[GS88]{GS88}
M.~Gros and N.~Suwa, 
\emph{La conjecture de Gersten pour les faisceaux Hodge-Witt 
logarithmique}, 
Duke Math. Journal \textbf{57} (1988), 
p.615-628.


\bibitem[Gro68]{Gro68}
A.~Grothendieck, 
\emph{Classes de Chern et representations lin\`eaires des groupes discrets}, 
Dix exposes sur la cohomologie des schemas, 
North Hollland (1968), 
p.215-305.

\bibitem[Har66]{Har66}
R.~Hartshorne, 
\emph{Residues and duality}, 
Lect. Notes in Math. \textbf{20}, 
Springer-Verlag, Berlin, Heidelberg, New York (1966).


\bibitem[Her97]{Her97}
I.~Herzog, 
\emph{{T}he {Z}iegler spectrum of a locally coherent {G}rothendieck category}, 
Proc. London Math. Soc. \textbf{74} (1997), 
p.503-558.

\bibitem[Hil81]{Hil81}
H.~Hiller, 
\emph{{$\lambda$}-rings and algebraic {$K$}-theory}, 
Jour. of pure and applied alg. \textbf{20} (1981), 
p.241-266.

\bibitem[HM10]{HM10}
T.~Hiranouchi and S.~Mochizuki, 
\emph{{P}ure weight perfect modules over divisorial schemes}, 
in {D}eformation Spaces: {P}erspectives on {A}lgebro-geometric {M}oduli 
(2010), 
p.75-89.

\bibitem[Hir03]{Hir03}
P.~Hirschhorn, 
\emph{Model categories and their localizations}, 
Mathematical Surveys and Monographs \textbf{99}, 
American Mathematical Society (2003).

\bibitem[Hov99]{Hov99}
M.~Hovey, 
\emph{Model categories}, 
Math. Survey and Monographs \textbf{63} 
American Math. Society (1999).

\bibitem[Hov01]{Hov01}
M.~Hovey, 
\emph{Spectra and symmetric spectra in general model categories}, 
J. Pure App. Alg. \textbf{165} (2001), 
p.63-127.


\bibitem[Iwa59]{Iwa59}
K.~Iwasawa, 
{\it{On $\Gamma$-extensions of algebraic number fields}}, 
Bull Amer. Math. Soc., 
\textbf{65} (1959), 
p.183-226. 

\bibitem[Kel06]{Kel06}
B.~Keller, 
\emph{On differential graded categories}, 
International Congress of Mathematicians (Madrid), 
Vol. II, Eur. Math. Soc., Z\"urich (2006), 
p.151-190. 

\bibitem[Ker09]{Ker09}
M.~Kerz, 
\emph{The Gersten conjecture for Milnor $K$-theory}, 
Invent. Math. \textbf{175} (2009), 
p.1-33.

\bibitem[KM16]{KM16}
A.~Krishna and M.~Morrow, 
\emph{Analogue of Gersten's conjecture for singular schemes}, 
arXiv:1508.05621v2 (2016).


\bibitem[Lev85]{Lev85}
M.~Levine, 
\emph{{A} {$K$}-theoretic approach to multiplicities}, 
Math. Ann. \textbf{271} (1985), 
p.451-458.

\bibitem[Lev88]{Lev88}
M.~Levine, 
\emph{Localization on singular varieties}, 
Invent.\ Math.\ \textbf{91}
  (1988), 
p.423-464.

\bibitem[Lev08]{Lev08}
M.~Levine, 
\emph{The homotopy coniveau tower}, 
J. Topol. \textbf{1} (2008), 
p.217-267. 

\bibitem[Lev13]{Lev13}
M.~Levine, 
\emph{Convergence of Voevodsky's slice tower}, 
Documenta Mathematica \textbf{18} (2013), 
p.907-941.

\bibitem[MVW06]{MVW06}
C.~Mazza, V.~Voevdosky and C.~Weibel, 
\emph{Lecture notes on motivic cohomology}, 
Clay Mathematics Monographs \textbf{2}, 
American Mathematical Society (2006).


\bibitem[Moc13a]{Moc13a}
S.~Mochizuki, 
\emph{{H}igher $K$-theory of Koszul cubes}, 
Homology, Homotopy and Applications Vol. \textbf{15} (2013), p. 9-51.

\bibitem[Moc13b]{Moc13b}
S.~Mochizuki, 
\emph{{N}on-connective {$K$}-theory of relative exact categories}, 
Preprint, available at arXiv:1303.4133 (2013).


\bibitem[Moc15]{Moc15}
S.~Mochizuki, 
\emph{On Gersten's conjecture slide movie}, 
avilable at\\
{\tt http://www.surgery.matrix.jp/math/ohkawa/indexj.html} (2015).

\bibitem[Moc16a]{Moc16a}
S.~Mochizuki, 
\emph{On Gersten's conjecture}, 
prepreint, avilable at 
arXiv:1503.07966v3 (2016).

\bibitem[Moc16b]{Moc16b}
S.~Mochizuki, 
\emph{Nilpotent invariant motives}, 
In preparation (2016). 


\bibitem[MY14]{MY14}
S.~Mochizuki and S.~Yasuda, 
\emph{What makes a multi-complex exact?}, 
avilable at arXiv:1305.6794v4 (2014).

\bibitem[Mor15]{Mor15}
M.~Morrow, 
\emph{A singular analogue of Gersten's conjecture and 
applications to $K$-theoretic 
ad\'eles}, 
Comm. Algebra \textbf{43} (2015), 
p.4951-4983.

\bibitem[Oja80]{Oja80}
M.~Ojanguren, 
\emph{Quadratic forms over regular rings}, 
Journal of Indian Math. Soc. \textbf{44} (1980), 
p.109-116.

\bibitem[OP99]{OP99}
M.~Ojangruen and I.~Panin, 
\emph{A purity theorem for the Witt group}, 
Ann. Scient. \'Ec. Norm. Sup. \textbf{32} (1999), 
p.71-86




\bibitem[Pan03]{Pan03}
I.~Panin, 
\emph{{T}he equi-characteristic case of the {G}ersten conjecture}, 
Proc. Steklov. Inst. Math. \textbf{241} (2003), 
p.154-163.

\bibitem[Par82]{Par82}
W.~Pardon, 
\emph{A ``Gersten conjecture'' for Witt groups}, 
Algebraic $K$-theory Part II (Oberwolfach, 1980), 
Lect. Notes in Math. \textbf{967}, 
Springer-Verlag (1982), 
p.300-315.


\bibitem[Pop86]{Pop86}
D.~Popescu, 
\emph{{G}eneral {N}\'eron desingularization and 
approximation}, 
Nagoya Math. Jour. \textbf{104} (1986), 
p.85-115.

\bibitem[Qui73]{Qui73} 
D.~Quillen, 
\emph{{H}igher algebraic {$K$}-theory {I}}, 
In Higher $K$-theories, 
Springer Lect. Notes Math. \textbf{341} (1973), 
p.85-147.

\bibitem[RS90]{RS90}
L.~Reid and C.~Sherman, 
\emph{{T}he relative form of Gersten's conjecture for power series over 
a complete discrete valuation ring}, 
Proc. Amer. Math. Soc. \textbf{109} (1990), 
p.611-613.

\bibitem[Ros08]{Ros08}
A.~L.~Rosenberg, 
\emph{{T}opics in noncommutative algebraic geometry, homological algebra 
and $K$-theory}, 
preprint MPI  (2008).

\bibitem[SZ03]{SZ03}
A.~Schmidt and K.~Zainoulline, 
\emph{Generic injectivity for \'etale cohomology and pretheories}, 
Journal of Algebra \textbf{263} (2003), 
p.215-227.


\bibitem[Seg74]{Seg74}
G.~Segal, 
\emph{{C}ategories and cohomology theories}, 
Topology \textbf{13} (1974), 
p.293-312.

\bibitem[Ser59]{Ser59}
J.~P.~Serre, 
{\it{Classes des corps cyclotomiques (d'apr\`es K. Iwasawa)}}, 
S\'em. Bourbaki, Expose \textbf{174} (1958-1959).



\bibitem[She78]{She78}
C.~Sherman, 
\emph{The $K$-theory of an equicharacteristic discrete valuation ring 
injects into the $K$-theory of its field of quotients}, 
Pacific J. Math \textbf{74} (1978), 
p.497-499.

\bibitem[She82]{She82}
C.~Sherman, 
\emph{{G}roup representations and algebraic {$K$}-theory}, 
In Algebraic $K$-theory, Part I, 
Lecture Notes in Math. \textbf{966}, 
Springer, Berlin (1982), 
p.208-243. 

\bibitem[She89]{She89}
C.~Sherman, 
\emph{$K$-theory of discrete valuation rings}, 
Jour. Pure Appl. Algebra \textbf{61} (1989), 
p.79-98.

\bibitem[Smo87]{Smo87}
W.~Smoke, 
\emph{{P}erfect {M}odules over {C}ohen-{M}acaulay local rings}, 
Journal of algebra \textbf{106} (1987), 
p.367-375.

\bibitem[Swa63]{Swa63}
R.~Swan, 
\emph{{T}he {G}rothendieck ring of a finite group}, 
Topology \textbf{2} (1963), 
p.85-110.

\bibitem[Tab05]{Tab05}
G.~Tabuada, 
\emph{Invariants additfs de dg-cat\'egories}, 
Int. Math. Res. Not. \textbf{53} (2005), 
p.3309-3339. 

\bibitem[Tab08]{Tab08}
G.~Tabuada, 
\emph{Higher $K$-theory via universal invariants}, 
Duke Math. J. \textbf{145} (2008), 
p.121-206.

\bibitem[Tat65]{Tat65}
J.~T.~Tate, 
\emph{Algebraic cycles and poles of zeta functions}, 
In: Arithmetical Algebraic Geometry, Pundue Univ., 1963, 
Haper \& Row, New York (1965), 
p.93-110.

\bibitem[Tat94]{Tat94}
J.~T.~Tate, 
\emph{Conjectures on algebraic cycles in $l$-adic cohomology}, 
In: Motives, Seattle, WA, 1991, 
Proc. Sympos. Pure Math., \textbf{55}, part 1, 
Amer. Math. Soc., Providence, RI (1994), 
p.71-83.



\bibitem[TT90]{TT90}
R.~W.~Thomason and T.~Trobaugh, 
\emph{{H}igher algebraic {$K$}-theory of schemes 
and of derived categories}, 
The Grothendieck Festschrift, Vol.\ III, Progr.
  Math., vol.~\textbf{88}, 
Birkh\"auser Boston, Boston, MA, (1990), 
p.247-435.

\bibitem[TV07]{TV07}
B.~To\"en and M.~Vaqui\'e, 
\emph{Moduli of objects in dg-categories}, 
Ann. Sci. de l'ENS \textbf{40} (2007), 
p.387-444. 

\bibitem[Voe00]{Voe00}
V.~Voevodsky, 
\emph{Cohomology theory of presheaves with transfers}, 
Cycles, transfers, and motivic homology theories, 
Ann. of Math. Stud. Vol.\textbf{143}, 
Princeton Univ. Press, Princeton (2000), 
p.87-137.


\bibitem[Wal78]{Wal78}
F.~Waldhausen, 
\emph{{A}lgebraic {$K$}-theory of generalized free products}, 
Ann. of Math. \textbf{108} (1978), 
p.135-256.

\bibitem[Wal85]{Wal85}
F.~Waldhausen, 
\emph{{A}lgebraic {$K$}-theory of spaces}, 
In Algebraic and geometric topology, 
Springer Lect. Notes Math. \textbf{1126} (1985), 
p.318-419.


\bibitem[Wal98]{Wal98}
M.~Walker, 
\emph{The primitive topology of a scheme},
Journal of algebra Vol.\textbf{201} (1998), 
p.656-685.

\bibitem[Wal00]{Wal00}
M.~Walker, 
\emph{Adams operations for bivariant $K$-theory 
and a filtration using projective lines}, 
$K$-theory \textbf{21} (2000), 
p.101-140.

\bibitem[Wei94]{Wei94}
C.~ A.~Weibel, 
\emph{{A}n introduction to homological algebra}, 
Cambridge studies in advanced mathematics \textbf{38}, 
(1994).


\end{thebibliography}
\end{document}